\documentclass[reqno, a4paper]{amsart}     
\usepackage[utf8x]{inputenc}
\usepackage{setspace}
\usepackage[left=3cm, right=3cm, bottom=3cm, top=3cm]{geometry}         

\usepackage{mathtools}
\usepackage{amsfonts}
\usepackage{amsmath}
\usepackage{amssymb}
\usepackage{amsthm}
\usepackage{enumitem}
\usepackage{xcolor}
\usepackage{mathrsfs}
\usepackage{bm}
\usepackage{float}
\usepackage[colorinlistoftodos,prependcaption,textsize=tiny]{todonotes}

\usepackage{hyperref}
\hypersetup{
  colorlinks = true,
  linkcolor = blue
}

\def\R{\mathbb{R}}

\def\Z{\mathbb{Z}}
\def\C{\mathbb{C}}

\def\SS{\mathbb{S}}
\def\T{\mathbb{R}/\mathbb{Z}}
\def\X{\mathbb{X}}
\def\P{\mathcal{P}}
\def\supp{{\rm supp}}
\def\F{\operatorname{F}}

\renewcommand{\d}{\text{\rm d}}
\renewcommand{\epsilon}{\varepsilon} %

\newcommand{\dD}{\partial\mathbb{D}}
\newcommand{\Sd}{\mathbb{S}^d}

\newcommand{\x}{{\bm x}}

\renewcommand{\v}{{\bm v}}
\newcommand{\xxi}{{\bm \xi}}

\newcommand{\wt}{\widetilde}   
\newcommand{\muarc}{\overset{\smallsmile}{\mu}}
\newcommand{\nuarc}{\overset{\smallsmile}{\nu}}
\newcommand{\lambdarc}{\overset{\smallsmile}{\lambda}}

\renewcommand{\r}{\operatorname{r}}

\theoremstyle{plain}
\newtheorem{theorem}{Theorem}

\newtheorem{corollary}[theorem]{Corollary}
\newtheorem{proposition}[theorem]{Proposition}
\newtheorem{lemma}[theorem]{Lemma}

\theoremstyle{definition}

\newtheorem*{definition*}{Definition}

\newtheoremstyle{remarkstyle}  %
  {\topsep}                    %
  {\topsep}                    %
  {\normalfont}                %
  {}                           %
  {}                           %
  {.}                          %
  { }                          %
  {\textsc{#1}#2\thmnote{(#3)}}  %

\theoremstyle{remarkstyle}
\newtheorem*{remark}{Remark}

\numberwithin{equation}{section}

\allowdisplaybreaks

\title[Sharp sign uncertainty for trigonometric polynomials]{Sharp sign uncertainty for trigonometric polynomials}

\author[Ismoilov]{Tolibjon Ismoilov}

\address{SISSA - Scuola Internazionale Superiore di Studi Avanzati, Via Bonomea 265, 34136 Trieste, Italy}
\email{tolibjon.ismoilov@sissa.it}
\email{tolibjon.iismoilov@gmail.com}

\date{\today}                                         
\begin{document}

\begin{abstract}
We study sign uncertainty principles for trigonometric polynomials of prescribed degree $N$ with respect to a symmetric Borel measure $\mu$ on the unit circle $\T$. For each such measure, we determine the smallest radius of the last sign change for trigonometric polynomials with non-positive $\mu$-integral. We further extend these results to polar measures on higher-dimensional spheres $\mathbb{S}^d$, showing that the extremal problem reduces to the one-dimensional case via the polar part of the measure, and we establish a polynomial analogue on $[0,1]$ using orthogonal polynomials on the real line. 
\end{abstract}

\maketitle

\section{Introduction}
\subsection{Background} 
The Fourier uncertainty principles in mathematics describe the inability to control a function and its Fourier transform simultaneously. A classical instance is the Heisenberg inequality, which asserts that the mass of a function $f\in L^2(\R^d)$ and that of its Fourier transform $\widehat{f}$ cannot be localized near the origin simultaneously; see \cite{FollandSitaram1997TheUncertaintyPrinciple}. Within the framework of Fourier uncertainty, there is a concept known as \emph{sign uncertainty}. This refers to the idea that, under certain conditions, it is impossible for both a function and its Fourier transform to have negative mass concentrated near the origin. Such a phenomenon was first observed by Bourgain, Clozel, and Kahane \cite{Bourgain2010PrincipeDHeisenbergFonctionsPositives} in 2010. To describe it concretely, we recall some terminology.

\smallskip

Throughout this paper, we will use the following convention for defining eventually non-negative functions. Let $\mathbb{X}$ be either the euclidean space $\R^d$, the sphere $\Sd$, the unit circle $\T$ or the unit interval $[0, 1]$ 
equipped with its standard metric $\mathbf{\mathsf{d}}(\cdot, \cdot)$. Let $x^\ast\in \X$ be a distinguished point to serve as the origin or the north pole. We say a function $f:\X \to \R$ is \emph{eventually non-negative} if there exists a radius $r>0$ such that $f(x)\geq 0$ for all $x\in \X\setminus B_r(x^\ast)$, where $B_r(x^\ast)\coloneqq\bigl\{y\in \X \; :\; \mathbf{\mathsf{d}}(x^\ast, y)<r  \bigr\}$. The smallest such radius is called \emph{the radius of last sign change} and denoted by 
\begin{equation}\label{eq:definition-of-r-of-f}
    \r(f)\coloneqq \inf\bigl\{r>0 \; :\; f(x)\geq 0 \text{ for all } x\in \X\setminus B_r(x^\ast)  \bigr\}.
\end{equation}
Recall that for a function $f\in L^1(\R^d)$ we define its Fourier transform as 
\begin{equation*}
    \widehat{f}(\xxi)\coloneqq\int_{\R^d} e^{-2\pi i \xxi\cdot \x} f(\x)\,\d \x.
\end{equation*}

\smallskip
In \cite{Bourgain2010PrincipeDHeisenbergFonctionsPositives}, the authors proved that if $f:\R^d\to\R$ is integrable, with $\widehat{f}$ real-valued and integrable, satisfies
$f(0)\leq 0$, $\widehat{f}(0)\leq 0$, and both $f$ and $\widehat{f}$ are eventually
non-negative, then the product of their radii of last sign change is bounded from below by a positive
absolute constant:
\begin{equation}\label{eq:Bourgain-Clozel-Kahane}
  \r(f)\cdot \r(\widehat{f}) \geq \mathbb{A}(d) > 0.
\end{equation}
Finding the sharp value of $\mathbb{A}(d)$ in \eqref{eq:Bourgain-Clozel-Kahane} is a notoriously difficult problem. The only dimension in which the exact sharp constant is known is $d=12$, where Cohn and Gonçalves \cite{CohnGoncalves2019-AnOptimalUncertaintyPrinciple} proved $\mathbb{A}(12)=\sqrt{2}$, exploiting the machinery of Viazovska on modular forms  \cite{Viazovska2017-TheSpherePackingProblemInDimension8}.

\smallskip

The sign uncertainty phenomenon is not confined to the Euclidean setting. Gon\c{c}alves, Oliveira e Silva, and Ramos \cite{Goncalves2023NewSignUncertaintyPrinciples} vastly generalized the framework of Bourgain--Clozel--Kahane by establishing sign uncertainty principles for a broad collection of operators and function spaces. Their results cover Fourier and Dini series on the circle, the Hilbert transform, the discrete Fourier and Hankel transforms, spherical harmonics, and Jacobi polynomials, among others. Related quantitative aspects of the Euclidean sign uncertainty principle were further studied in \cite{CarneiroQuesada-Herrera2023GeneralizedSign,GoncalvesOliveiraESilvaRamos2021-OnRegularityAndMass, GoncalvesOliveira-e-SilvaSteinerberger-HermitePolynomialsLinearFlows}.

\smallskip
 
A central observation motivating the present work is that the underlying measure plays a decisive role in determining whether sign uncertainty can occur at all. This was made precise in the recent work of Carneiro, Ramos, and the author in \cite{Carneiro2024SignUncertaintyBrangesSpaces}, where they studied sign uncertainty for entire functions of exponential type with a competing condition expressed as integration with respect to an even Borel measure $\nu$ on $\R$. They established that certain choices of $\nu$ can eliminate the sign uncertainty phenomenon. Conversely, for a broad class of measures, they established sharp sign uncertainty bounds by bridging the problem and the theory of de Branges spaces of entire functions, where the reproducing kernel structure allows one to identify the sharp constants and classify the extremizers explicitly. 

\smallskip

The present article takes up the sign uncertainty phenomenon in the framework of \cite{Carneiro2024SignUncertaintyBrangesSpaces} to the compact setting of the unit circle $\T$, where entire functions of exponential type are replaced by \emph{trigonometric polynomials} of a prescribed degree, and the competing condition is given by a symmetric finite Borel measure $\mu$ on $\T$. Our main results provide a complete characterization of the sign uncertainty phenomenon for each measure and determine the sharp constants in the cases where sign uncertainty holds, complementing the body of work described above.

\subsection{An extremal problem for trigonometric polynomials} 

We identify the circle $\SS^1$ and the quotient space $\T$ in such a way that the distinguished point $x^\ast\in\SS^1$ corresponds to $0\in \T$. Let $\mu$ be a finite Borel measure on $\T$, which satisfies the symmetry condition
\begin{equation}\label{eq:symmetry-condition}
    \mu(A)=\mu(-A), \text{ for all } A\subseteq \T.
\end{equation}

We consider the following problem:
\begin{enumerate}[label=\textbf{(P\arabic*)}, ref=(P\arabic*)]
    \item \label{prob:trig-polynomials}
    For each fixed $N\geq 1$,  find the exact value of 
    \begin{equation*}
        \rho(\mu, N)\coloneqq \inf_{0\not \equiv f\in \mathcal{T}(\mu, N) } \r(f),
    \end{equation*}
    where the infimum is taken over the following family of trigonometric polynomials
    \begin{equation*}
        \mathcal{T}(\mu, N)\coloneqq\left\{ f(x)=\sum_{n=-N}^N \widehat{f}(n) \, e^{2\pi i n x}\; \left\lvert\; 
        \begin{aligned}
        & f:\T\to\R \text{ is eventually non-negative}\\
        &\text{ and } \int_{\T} f(x)\, \d \mu(x)\leq 0  
        \end{aligned}
        \right.
        \right\}.
    \end{equation*}
\end{enumerate}
Note that the metrics on the circle $\SS^1$ and on $\T$ differ by a factor of $2\pi$, and in the above problem, we are dealing with the radius of the last sign change on $\T$. The following observations allow us to restrict our attention to a smaller family of trigonometric polynomials.
\begin{remark}\label{rem: general remarks}\hfill
    \begin{enumerate}
        \item Note that  $\rho(c\mu, N)=\rho(\mu, N)$ for any $c>0$ and $N\geq 1$. Thus, it suffices to work with probability measures;
        \item The even part $f_e(x)\coloneqq \tfrac{f(x)+f(-x)}{2}$ of a function $f\in \mathcal{T}(\mu, N)$ satisfies $f_e\in \mathcal{T}(\mu, N)$ and $\r(f_e)\leq \r(f)$;
        \item If $f\in \mathcal{T}(\mu, N)$, then the function $f_0(x)= f(x)-\int_{\T} f(t)\,\d\mu(t)$ also belongs to the class $\mathcal{T}(\mu, N)$ and $\r(f_0)\leq \r(f)$; 
    \end{enumerate}
\end{remark}

Consider the case where $\d x$ is the Lebesgue measure on $\T$. Take a non-trivial function $f\in \mathcal{T}(\d x , N)$ and let $f_{\pm}(x)=\max\{\pm f(x), 0\}$ denote the positive and negative parts of $f$. Since $f$ has a non-positive integral, we have
\begin{equation*}
\begin{split}
\lVert f\rVert_{L^1}\leq  2\int_{\T}f_{-}(x)\, \d x&=2\int_{-\r(f)}^{\r(f)} f_{-}(x)\,\d x\leq 4 \r(f) \lVert f\rVert_{L^\infty}\\
&\leq 4(2N+1) \r(f) \max_{|n|\leq N} \lvert \widehat{f}(n)\rvert \leq 4(2N+1) \r(f)  \lVert f\rVert_{L^1}.
\end{split}
\end{equation*}
It implies that $\r(f)\geq 1/(4(2N+1))$ for any $f\in \mathcal{T}(\d x, N)$. Clearly, it shows the presence of the sign uncertainty phenomenon for the Lebesgue measure. In fact, this particular case of the problem has been considered in works \cite{Babenko1984ExtremalProblemPolynomials, YudinTwoExternalProblemsTrigonometric1996}.  In this paper, we are going to establish a more general framework that allows us to identify the exact value of $\rho(\mu, N)$ for any finite, even Borel measure $\mu$ on $\T$. In particular, we will see that $\rho(\d x, N)=1/(2N+2)$.

\medskip

\subsection{Orthogonal polynomials on the unit circle}\label{subsec: opuc basics}
To state our main results, we recall several important concepts from the classical theory of orthogonal polynomials on the unit circle (OPUC). The theory of OPUC will play a key role in our results, as the theory of de Branges spaces \cite{deBranges1968-HilbertSpacesOfEntireFunctions} was crucial in finding the smallest radius of last sign change for bandlimited functions \cite{Carneiro2024SignUncertaintyBrangesSpaces}. Here we follow a classic reference \cite{SimonPart1OrthogonalPolynomialsUnitCircle2005} on the topic. A more detailed discussion on this topic will be presented in $\mathsection \ref{subsec: opuc in detail}$.

\subsubsection{Setup} We denote by $\mathbb{D}$ the unit disk $\{z\in \C: |z|<1\}$ and by $\dD$ the unit circle $\{z\in \C: |z|=1\}$. Throughout the manuscript, we use the parametrization $\{e^{2\pi i \theta}: \theta \in \T\}$ of the unit circle $\dD$. Let $\mu$ be a probability measure on $\dD$ (equivalently on $\T$) and set
\begin{equation*}
    \int_{\dD} f(z)\,\d\mu(z)\coloneqq \int_{\T} f(e^{2\pi i \theta}) \,\d\mu(\theta),
\end{equation*}
for any measurable function $f:\dD\to \C$. We say a probability measure $\mu$ on $\dD$ is \emph{trivial} if its support is finite, i.e.  $\#\supp(\mu)<\infty$. 

\subsubsection{Orthogonal polynomials} For a probability measure $\mu$ on $\dD$ consider the Hilbert space $L^2(\dD, \d\mu)$ equipped with the inner product 
\begin{equation*}
    \langle f, g\rangle\coloneqq \int_{\dD} f(z)\overline{g(z)}\,\d\mu(z).
\end{equation*}
Since any bounded function on $\dD$ belongs to $L^2(\dD, \d\mu)$, we know that polynomials $1, z, z^2, \ldots$ lie in this Hilbert space. As long as the measure $\mu$ is non-trivial, by applying the Gram--Schmidt process, one can construct the unique \emph{monic orthogonal polynomials}  $\Phi_n(z)\coloneqq\Phi_n(z; \d\mu)$, for all $n\geq 0$,  in $L^2(\dD, \d\mu)$. Clearly by definition $\Phi_0(z)\equiv 1$ and we have $\Phi_n(z)=z^n+\text{lower order terms}$, and
\begin{equation*}
    \int_{\dD} \Phi_n(z)\, \overline{z^j}\,\d\mu(z)=0, \quad 0\leq j<n.
\end{equation*}
We also consider the sequence of \emph{orthonormal polynomials} $\varphi_n(z)\coloneqq \Phi_n(z)/\lVert \Phi_n\rVert_2$. 

\smallskip

Note that for a trivial measure $\mu$, with $\#\supp(\mu)=M$, one can construct the polynomials $\Phi_0, \Phi_1, \ldots, \Phi_M$ uniquely by the same procedure, and the process terminates exactly at $\Phi_M$, since it vanishes at every point in the support of $\mu$. Thus, the only well-defined orthonormal polynomials are $\varphi_0, \varphi_1, \ldots, \varphi_{M-1}$. 

\subsubsection{Auxiliary polynomials}
Let $P(z)$ be a polynomial of degree $n\geq 0$, given by 
\begin{equation*}
    P(z)=a_n z^n+a_{n-1}z^{n-1}+\ldots +a_1 z+a_0.
\end{equation*} 
For an integer $m\geq n$ define the \emph{reversed polynomial} $P^{*, m}(z)$ by 
\begin{equation*}
    P^{*, m}(z)\coloneqq z^m \overline{P(1/\overline{z})}=\overline{a_0} z^m+\overline{a_{1}}z^{m-1}+\ldots +\overline{a_{n-1}} z^{m-n+1}+\overline{a_n}z^{m-n}.
\end{equation*}
For simplicity we denote $P^{*}=P^{*, n}$.

\smallskip

Whenever $\mu$ is a probability measure and the unique orthonormal polynomials $\varphi_n(z; \d\mu)$ exist, then one has 
\begin{equation}\label{eq:hermite-biehler in the intro}
    \lvert\varphi_n(z)\rvert < \lvert \varphi_n^*(z)\rvert
\end{equation}
for all $z\in \mathbb{D}$ (see Lemma \ref{lem:properties of orthogonal polynomials} below). Based on these facts, we define the following auxiliary polynomials. Whenever $\varphi_n(z; \d\mu)$ is well-defined as in the procedure described above, we set 
\begin{equation*}
    A_{n}(z;\d\mu)=\frac{\varphi_n(z;\d\mu)+\varphi_n^*(z;\d\mu)}{2}.
\end{equation*}
In the particular case of a trivial measure $\mu$ with $\#\supp(\mu)=n$, we use the following definition instead 
\begin{equation*}
    A_n(z;\d\mu)=\frac{z \varphi_{n-1}(z;\d\mu)+\varphi_{n-1}^\ast(z;\d\mu)}{2}.
\end{equation*}
Later in $\mathsection \ref{subsec: rkhs of polynomials}$, we will provide a complete justification for these polynomials and for their role in the extremal problem \ref{prob:trig-polynomials}. Note that \eqref{eq:hermite-biehler in the intro} implies that $A_n(z)$ has no zeros in the unit disk. On the other hand, the symmetry $A_n=A^*_n$ implies all zeros of $A_n$ lie on $\dD$.

\subsubsection{Main result}
Having introduced all the components, we state our main result concerning the extremal problem \ref{prob:trig-polynomials}. We split the statement of the result into two cases based on the support of the measure.

\begin{theorem}\label{thm:best-radius-of-negativity}
    Let $\mu$ be a finite Borel measure on $\T$ which satisfies \eqref{eq:symmetry-condition}. The following holds
    \begin{enumerate}[label=\textnormal{(\roman*)}]
        \item If  $1\leq N<\#\supp(\mu)$, then we have 
    \begin{equation*}
        \rho(\mu, N)=\theta_1,
    \end{equation*}
    where $0\leq \theta_1\leq 1/2$ is such that $e^{2\pi i \theta_1}$ is a zero of $A_{N+1}(z;\d\mu)$ which is at a minimal distance from the point $1\in\dD$, i.e., 
    \begin{equation*}
        \theta_1=\min\{|\theta|\, : \, A_{N+1}(e^{2\pi i \theta})=0\}.
    \end{equation*}
    Moreover, the extremizer is unique (up to a constant factor) and given by 
    \begin{equation}\label{eq:extremizer-function}
       f(x)=\frac{|A_{N+1}(e^{2\pi i x})|^2}{\cos(2\pi \theta_1)-\cos(2\pi x)}.
    \end{equation}
    \item If $\#\supp(\mu)\leq N<\infty$, then we have $\rho(\mu, N)=0$.
    \end{enumerate}
    
\end{theorem}

\medskip

\subsection{Analogues in higher dimensions}
Throughout this manuscript, vectors in $\R^{d+1}$ are written in bold font 
(e.g., $\bm{x}, \bm{y}, \bm{z}$), while scalars in $\R$ or $\C$ are written in regular font (e.g., $x, y, z$). Let $d\geq 2$  and consider the $d$-dimensional sphere $\Sd \subset \R^{d+1}$ and let $\bm{v} \in \Sd$ be a distinguished point. Without loss of generality we can assume $\bm{v}=(0, 0, \ldots, 0 , 1)$. We denote by $\mathrm{SO}(d+1)$ the special orthogonal group, consisting of all $(d+1) \times (d+1)$ real orthogonal matrices with determinant equal to $1$, and by $\mathrm{SO}_{\bm{v}}(d+1) \subsetneq \mathrm{SO}(d+1)$ the stabilizer subgroup of $\bm{v}$, namely the closed subgroup of all rotations in $\mathrm{SO}(d+1)$ that fix $\bm{v}$.

\subsubsection{Polar measures}\label{subsec:polar-measures} We say a finite Borel measure $\mu$ is a  \emph{polar measure (with a pole at $\v$)} on $\Sd$ if it satisfies the following symmetry property: 
\begin{equation}\label{eq:rotation invariance}
    \mu(A)=\mu(R(A)), \quad \text{ for all } A\subseteq \Sd \text{ and for all } R\in \mathrm{SO}_{\bm{v}}(d+1),
\end{equation}
where $R(A)\coloneqq\{R x\, | \, x\in A\}$. For a polar measure $\mu$ we set its \emph{polar part} to be measure $\mu_p$ on $\T$ given by
\begin{equation*}
    \mu_p(A)\coloneqq \mu\left( A^{+, d}  \right)+\mu\left( A^{-, d}  \right) + 2 \,\mu(A^{d}\cap\{\v, -\v\}),
\end{equation*}
where the lifts $A^d$ and $ A^{\pm, d} $ are given by
\begin{align}
 A^d &\coloneqq\left\{ \bigl(\x^\prime \sin(2\pi \theta), \cos(2\pi \theta)\bigr) \, :\, \theta\in A, \; \x^\prime\in \mathbb{S}^{d-1} \right\}\label{eq:liftofasetA}
 \intertext{and}
A^{\pm, d} & \coloneqq\left\{ \bigl(\x^\prime \sin(2\pi \theta), \cos(2\pi \theta)\bigr) \, :\, \pm\theta\in A\cap (0, 1/2), \; \x^\prime\in \mathbb{S}^{d-1} \right\}.\nonumber
\end{align}
The set of all polar measures has a simple characterization in terms of their polar parts; see Proposition \ref{prop:decomposition-for-polar-measures}.
For example, the polar part of the normalized surface measure $\sigma_{d}$ is given by
\begin{equation*}
    \d\mu_{d-1}(\theta)\coloneqq\d(\sigma_{d})_p(\theta)=C_d \, \lvert \sin{(2\pi \theta)} \rvert^{d-1} \d\theta,
\end{equation*}
where $C_d>0$ is the dimensional constant. Note that the polar part of a polar measure $\mu$ is always an even measure on $\T$.

\subsubsection{Extremal problem in higher dimensions}

Denote by $\mathscr{H}_k$ the space of spherical harmonics of degree $k$ on $\mathbb{S}^d$. Each $\mathscr{H}_k$ is a finite-dimensional subspace of $L^2(\mathbb{S}^d, \d\sigma_d)$, and the span of all spherical harmonics is dense in $L^2(\mathbb{S}^d, \d\sigma_d)$; see \cite[Chapter IV]{SteinWeissFourierAnalysis}.
\medskip

Let $\mu$ be a polar measure with a pole at $\v$ on $\Sd$ and $N\geq 1$ be an integer. Consider the family $\mathcal{T}(d, \mu, N)$ of functions $f:\Sd \to \R$ that satisfy the following conditions:
\begin{enumerate}
    \item $f(\bm{x})=\sum_{k=0}^N Y_k(\bm{x})$, where $Y_k\in \mathscr{H}_k$;
    \item $f$ is eventually non-negative on $\Sd$;
    \item $\displaystyle \int_{\Sd} f(\bm{x})\, \d\mu(\bm{x})\leq 0$.
\end{enumerate}
We are interested in finding the function in this class that has the negative part concentrated in the smallest possible spherical cap, which is a ball in the geodesic metric of $\Sd$. We formulate this problem as follows: 
\begin{enumerate}[label=\textbf{(P\arabic*)}, ref=(P\arabic*), start=2]
    \item \label{prob:spherical-harmonics} For each $N\geq 1$ find the exact value of
    \begin{equation*}
    \rho(d, \mu, N)\coloneqq \inf_{0\not \equiv f \in \mathcal{T}(d, \mu, N)} \r(f),
    \end{equation*}
    where $\r(f)$ is defined as in \eqref{eq:definition-of-r-of-f} with respect to the usual metric of the $\Sd$.
\end{enumerate}

\noindent Note that the $d=1$ case of this problem can be considered, and it will be the same as the problem \ref{prob:trig-polynomials} up to a constant factor of $2\pi$ due to the differences between the metrics of $\T$ and $\SS^1$.

\medskip

Our next result allows us to say that the nature of the problem \ref{prob:spherical-harmonics} does not depend on the dimension, in the following sense. 

\begin{theorem}\label{thm:reduction-to-dimension-one}
    For every polar measure $\mu$ on $\Sd$ and for each integer $N\geq 1$, we have
    \begin{equation*}
        \rho(d, \mu, N)= 2\pi\,\rho(\mu_p, N),
    \end{equation*}
    where $\mu_p$ is the polar part of the measure $\mu$. The function $\widetilde{f}$, given by
    \begin{equation*}
        \wt f(\x)= f\left(\frac{1}{2\pi} \arccos(\x\cdot\v)\right),
    \end{equation*}
is an extremizer to the problem \textnormal{\ref{prob:spherical-harmonics}}, where $f:\T\to \R$ is defined in \eqref{eq:extremizer-function}.
\end{theorem}

\subsection{Analogues for polynomials}\label{sec:extremal problem for polynomials}

We now turn to the polynomial analogue of the sign uncertainty phenomenon considered in problem \ref{prob:trig-polynomials}.

\subsubsection{Setup} Let $\mu$ be a finite Borel measure on $[0, 1]$, and fix $N \geq 1$.
Define the family of real polynomials of degree at most $N$ by
\begin{equation*}
    \mathcal{P}(\mu, N) \coloneqq \left\{ p(x) = \sum_{n=0}^N a_n x^n \;\Bigg|\;
    \begin{aligned}
        & \; p:[0, 1]\to \R \text{ is eventually non-negative}, \\
        & \; \int_{0}^{1} p(x)\, \d\mu(x) \leq 0
    \end{aligned}
    \right\}.
\end{equation*}
Note that in this case the distinguished point is $x^\ast=0$. In other words, $\mathcal{P}(\mu, N)$ consists of polynomials that have non-positive $\mu$-integral and are non-negative on some left neighborhood of the point $x=1$. Again, we study the smallest possible radius of the last sign change using the following extremal problem.
\begin{enumerate}[label=\textbf{(P\arabic*)}, ref=(P\arabic*), start=3]
    \item \label{prob:polynomials}
    For a given integer $N \geq 1$, find the exact value of
    \begin{equation*}
        \Omega(\mu, N) \coloneqq \inf_{0 \not\equiv f \in \mathcal{P}(\mu, N)} \r(f),
    \end{equation*}
    where $\r(f)$ is given by \eqref{eq:definition-of-r-of-f}.
\end{enumerate}
We can also address this problem using the classical theory of orthogonal polynomials.
\subsubsection{Orthogonal polynomials on the real line (OPRL)} Given a finite Borel measure $\mu$ supported on $[0, 1]$, we consider the Hilbert space $L^2(\R, \d\mu)$ equipped with the inner product 
\begin{equation*}
    \langle f, g\rangle= \int_{\R} f(x)\overline{g(x)}\,\d\mu(x).
\end{equation*}
As before, we have $1, x, x^2, \ldots \in L^2(\R, \d\mu)$ hence, the Gram--Schmidt orthogonalization process leads us to the set of unique monic orthogonal polynomials $\Psi_0, \Psi_1, \Psi_2 \ldots, \Psi_n(x):=\Psi_n(x;\d\mu), \ldots $. Again, the process terminates at $\Psi_m(x;\d\mu)$ if and only if $\#\supp(\mu)=m$, and in such a case the polynomial $\Psi_m$ vanishes at all points in the support of $\mu$. For a more detailed discussion on OPRL, we refer the reader to the classical literature on the topic, such as \cite{Ismail2005ClassicalAndQuantum, SzegoOrthogonalPolynomials1939}.

Our next result concerns the exact value of $\Omega(\mu, N)$ stated in the problem \ref{prob:polynomials}. 

\begin{theorem}\label{thm:sign uncertainty for polynomials}    
    Let $\mu$ be a Borel probability measure on $[0, 1]$ with infinite support. Let $n\geq 1$ be an integer and let $\mu_1$ be a measure given by $\d\mu_1(x)=(1-x)\d\mu(x)$. Then we have the following 
    \begin{enumerate}[label=\upshape(\roman*)]
        \item $\Omega(\mu, 2n-1)=x_{n, 1}$, where $x_{n,1}$ is the smallest zero of $\Psi_{n}(x; \d\mu)$. Moreover, any extremizer must be of the form  
        \begin{equation*}
            p(x)=C\frac{\bigl(\Psi_n(x;\d\mu)\bigr)^2}{x-x_{n,1}}, \quad C>0.
        \end{equation*}
        \item $\Omega(\mu, 2n)=\xi_{n, 1}$, where $\xi_{n,1}$ is the smallest zero of $\Psi_{n}(x; \d\mu_1)$. Moreover, any extremizer must be of the form  
        \begin{equation*}
            p(x)=C(1-x)\frac{\bigl(\Psi_{n}(x; \d\mu_1)\bigr)^2}{x-\xi_{n,1}}, \quad C>0.
        \end{equation*}
    \end{enumerate}
\end{theorem}

\begin{remark}
With a slight modification in the statement and in the arguments, this result can be extended to measures with finite support.
\end{remark}

It was first observed by Szeg\"{o} in 1921 \cite{Szego1921} that OPUC and OPRL
are not completely independent theories for measures related in a specific way.
Indeed, one can establish a relation between the problems \ref{prob:trig-polynomials} and \ref{prob:polynomials}, see Proposition \ref{prop:equivalence-with-polynomial-problem}.
However, we choose to present both proofs independently, as each argument stands
on its own. More importantly, our result addresses problem \ref{prob:polynomials}
directly in terms of the measures $\mu$ and $\mu_1$, without passing through
the trigonometric polynomial setting of problem \ref{prob:trig-polynomials} via Proposition \ref{prop:equivalence-with-polynomial-problem}.

\subsection{Applications}

\subsubsection{The Lebesgue measure on \texorpdfstring{$\T$}{R/Z}} Let $\d x$ be the usual Lebesgue measure on $\T$ with the normalization $\int_{\T} \d x=1$. In this case, it is easy to see that the orthogonal polynomials on the unit circle are given by $\Phi_n(z; \d x)=z^n$, $n\geq 0$. Moreover, $\varphi_n=\Phi_n$ for all $n\geq 0$.
Hence, the value of $\rho(\d x, N)$ is determined by the zeros of the polynomial $A_{N+1}(z)=(z^{N+1}+1)/2$ and by Theorem \ref{thm:best-radius-of-negativity} we know
\begin{equation*}
    \rho(\d x, N)=\frac{1}{2\pi}\frac{\pi}{N+1}=\frac{1}{2(N+1)}.
\end{equation*}
The extremizers of this problem have the following form
\begin{equation*}
     f(x)=C\frac{|e^{2\pi i (N+1) x}+1|^2}{4(\cos(\pi/(N+1))-\cos(2\pi x))}=\frac{C\cos^2(\pi (N+1)x)}{\cos(\pi/(N+1))-\cos(2\pi x)},
\end{equation*}
where $C>0$ is a positive constant. In Figure \ref{fig:extremizers}, we show plots of some of these functions.
\begin{figure}[ht]
    \centering
    \includegraphics[width=0.7\linewidth]{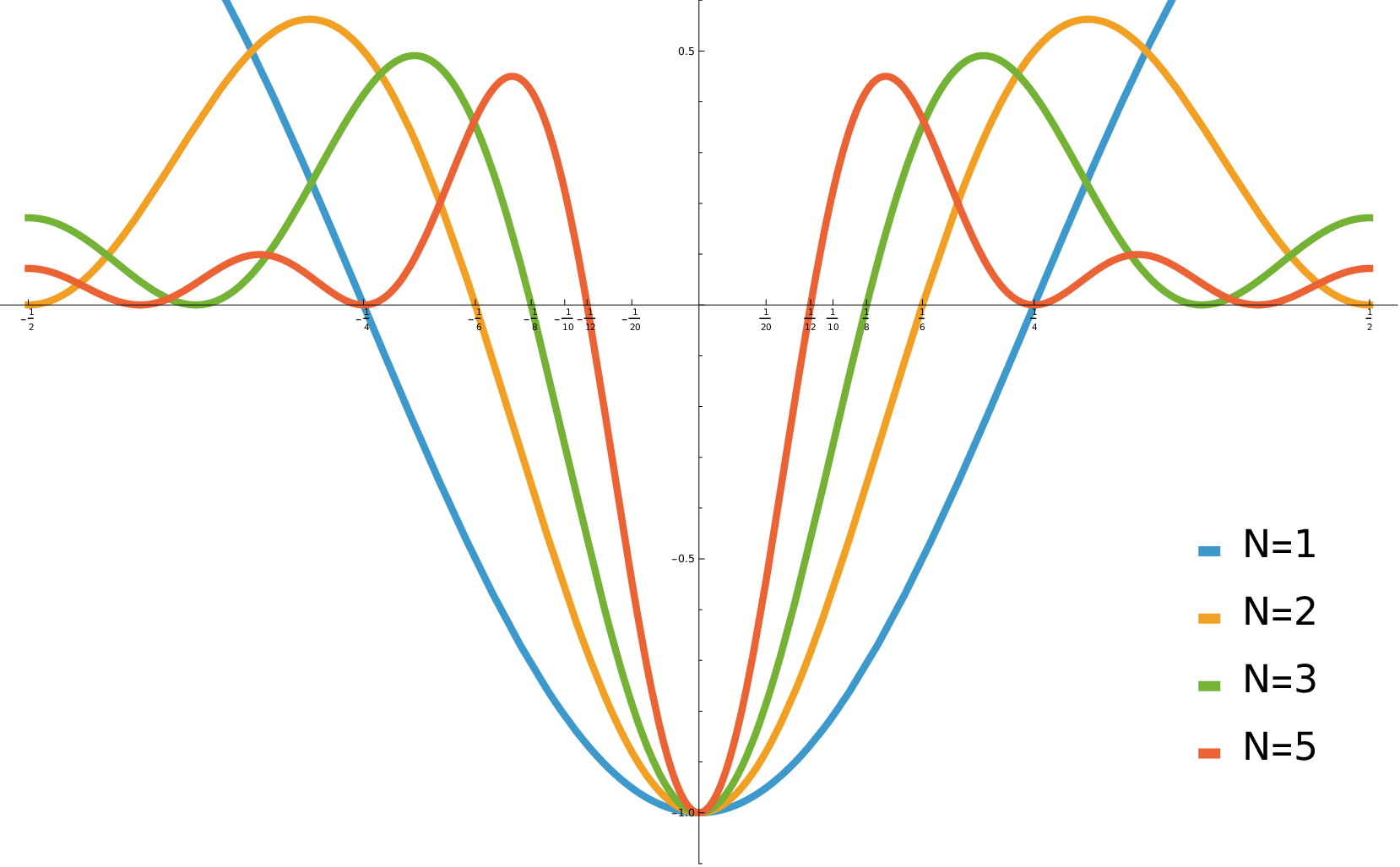}
    \caption{Maximizers for the problem \texorpdfstring{$\rho(\d x, N)$}{} when $\d x$ is the Lebesgue measure.}
   \label{fig:extremizers}
\end{figure}

\subsubsection{Uniform measures on higher-dimensional spheres}
Let $\sigma_d$ be the normalized surface measure on $\Sd$, and by Proposition \ref{prop:decomposition-for-polar-measures} we know its polar part is the measure $\mu_{d-1}$ given by 
\begin{equation}\label{eq:polar-part-of-surface-measure}
    \d\mu_{d-1}(\theta)=C_d |\sin(2\pi \theta)|^{d-1}\,\d\theta.    
\end{equation}
As a consequence of Theorems \ref{thm:best-radius-of-negativity}, \ref{thm:reduction-to-dimension-one},  and \ref{thm:sign uncertainty for polynomials}, we conclude the following. 

\begin{corollary}\label{cor:uniform-measures-on-spheres}
    Let $d \geq 2$, $n \geq 1$, and set $\alpha = (d-2)/2$. Let $\sigma_d$ denote the normalized surface measure on $\Sd$. Then
    \begin{enumerate}[label=\textnormal{(\roman*)}]
        \item $ \rho(d, \sigma_d, 2n-1)=2\pi \rho(\mu_{d-1}, 2n-1)=\arccos{(x_{n, n})}$, where $x_{n,n}$ is the largest zero of Jacobi polynomial $P^{(\alpha, \alpha)}_n(x)$;
        \smallskip
        \item $ \rho(d, \sigma_d, 2n)=2\pi \rho(\mu_{d-1}, 2n)=\arccos{(\xi_{n, n})}$, where $\xi_{n,n}$ is the largest zero of Jacobi polynomial $P^{(\alpha, \alpha+1)}_n(x)$.
    \end{enumerate}
\end{corollary}

\subsubsection{Application to a problem of Yudin}The previous corollary can be applied to study the problem about the measure of the negativity sets of harmonic polynomials posed by Yudin in \cite{YudinTwoExternalProblemsTrigonometric1996}. The problem asks the following:

\begin{quote}
\textit{If $f(\x)$ $(\x \in \mathbb{R}^{d+1})$ is a harmonic polynomial of degree at most $k$ such that $f(0) = 0$, then
\begin{equation*}
    \int_{\mathbb{S}^{d}} f(\x)\, \d\sigma_d(\x) = 0
\end{equation*}
by the mean value theorem. Hence $f$ assumes both positive and negative values on $\mathbb{S}^{d}$. What is the smallest possible value of
\begin{equation*}
\sigma_d(\{\x \in \mathbb{S}^{d} : f(\x) > 0\})?
\end{equation*}}
\end{quote}
Although the original work of Babenko \cite{Babenko1984ExtremalProblemPolynomials} addressed this problem in dimension one ($d=1$), it is a very delicate problem in higher dimensions. Recent progress on this and other related problems can be found in \cite{AndreevYudin2003PositiveValues, Yudin2004Onpositivevaluesofsphericalharmonics} and references therein.

\begin{corollary}\label{cor:yudins-problem}
    Let $0\leq \tau_k(d)$ be the answer to the question of Yudin, then we know 
    \begin{equation*}
        \tau_{2n-1}(d) \leq \sigma_d\bigl(\mathcal{C}(x_{n, n})\bigr)\quad\text{ and }\quad \tau_{2n}(d) \leq \sigma_d\bigl(\mathcal{C}(\xi_{n, n})\bigr),
    \end{equation*}
    where $\mathcal{C}(t)\coloneqq\{\x\in\Sd: \x\cdot\v\geq t\}$, $-1\leq t\leq 1$, is a spherical cap centered at $\v$  and $x_{n,n}, \xi_{n,n}$ are as in Corollary \ref{cor:uniform-measures-on-spheres}. 
\end{corollary}
Observe that Corollary \ref{cor:yudins-problem} is a direct consequence of Corollary \ref{cor:uniform-measures-on-spheres}.

\subsubsection{Dual problem for quadrature nodes}\label{subsec:dual problem for quadratures}
As before, let $\mu$ be an even probability measure on $\T$.  Recall that a finite set $\xi = \{\xi_1, \xi_2, \ldots, \xi_{N}\} \subsetneq \T$ is a \emph{set of nodes} of a quadrature if there exist weights $\lambda_j \in (0,1]$, $1\leq j\leq N$,  such that the identity
\begin{equation}\label{eq:n-point-quadrature}
    \int_{\T} f(x)\,\d\mu(x) = \sum_{j=1}^{N} \lambda_j\, f(\xi_j)
\end{equation}
holds for any trigonometric polynomial $f(x)=\sum_{|n|\leq N-1} \widehat{f}(n) e^{2\pi i n x}$. The identity \eqref{eq:n-point-quadrature} is called an $N$-\emph{point quadrature formula} for $\mu$. A more detailed overview of quadrature formulas can be found in \cite{JonesNjastadThronMomentTheoryOnUnitCircle} and references therein.  The following extremal problem for the set of nodes becomes the dual problem for \ref{prob:trig-polynomials}.

\begin{enumerate}[label=\textbf{(P\arabic*)}, ref=(P\arabic*), start=4]
    \item \label{prob:quadrature-formulas}
    Given $\mu$ and $N \geq 1$, find the exact value of
    \begin{equation*}
        \Lambda(\mu, N) \coloneqq \sup_{\#\xi=N}
        \min\bigl\{|\xi_j| : \xi_j \in \xi \bigr\},
    \end{equation*}
    where the supremum is taken over all sets of nodes $\xi$ with $N$ points. If the set is empty, we assume the supremum is 0.
\end{enumerate}

\begin{theorem}\label{thm:equivalence-with-quadrature-problem}
    Given $\mu$ and $N \geq 1$, the problems \textnormal{\ref{prob:trig-polynomials}} and \textnormal{\ref{prob:quadrature-formulas}} are related via the identity 
    \begin{equation*}
        \rho(\mu, N)=\Lambda(\mu, N+1).
    \end{equation*}
\end{theorem}
\section*{Notation}
For any set $A$, we denote its cardinality by $\# A$. For each $n\geq 1$, we denote by $\mathcal{P}_n$ the set of all polynomials of degree at most $n$. If $X$, $Y$ are two metric spaces, and $\mu$ is a measure on $X$, and $T : X \to Y$ is a Borel map, we denote by $T_{\#}\mu$ the \emph{push-forward of $\mu$ through $T$}, defined by
\begin{equation*}
    T_{\#}\mu(B) := \mu(T^{-1}(B)),
\end{equation*}
for all Borel sets $B\subseteq Y$. If a measure $\mu$ is defined on $\T$, then we set $\muarc$ as the ``upper-half'' of $\mu$ which is a measure on the interval $[0, 1/2]$ and is given by
\begin{equation*}
    \muarc(A)\coloneqq \mu(A\cap (0, 1/2))+ \frac{1}{2}\mu(A\cap \{0, 1/2\}),
\end{equation*}
for every $A\subseteq[0, 1/2]$.

\section{Dimension shift: Proof of Theorem \ref{thm:reduction-to-dimension-one}}\label{sec:dimension-shift}
The following result establishes the complete characterization of polar measures on $\Sd$.
We were unable to find an exact reference for it, so for the convenience of the reader, we provide a brief proof.
\begin{proposition}\label{prop:decomposition-for-polar-measures}
    If $\mu$ is a polar measure on $\Sd$, then it can be decomposed as $\mu =\muarc_p \otimes \sigma_{d-1}$, in the following sense
    \begin{equation*}\label{eq:disintegration identity as defintion}
        \int_{\Sd} f(\bm x)\, \d\mu(\bm x)= \int_{0}^{1/2}\left(\int_{\mathbb{S}^{d-1}} f\bigl(\x'\sin(2\pi \theta), \cos(2\pi \theta)\bigr)\,\d\sigma_{d-1}(\bm x')\right)\, \d\muarc_p(\theta),
    \end{equation*}
    where $\sigma_{d-1}$ is the normalized uniform surface measure on $\mathbb{S}^{d-1}$, i.e., $\sigma_{d-1}(\SS^{d-1})=1$, and $\mu_p$ is the polar part of $\mu$.

\end{proposition}
\begin{proof}
   We parametrize the sphere $\Sd$ using the polar angle and parallel latitudes: 
\begin{equation*}
    \Sd=\left\{\x=(\x'\sin(2\pi \theta), \cos(2\pi \theta)) \,:\, \x'\in \mathbb{S}^{d-1}, 0\leq \theta\leq 1/2\right\}. 
\end{equation*}
Consider the map $\operatorname{F}:\Sd\to [0, 1/2]$
\begin{equation*}
    \F(\x)=\frac{1}{2\pi} \arccos(\x\cdot\v)
\end{equation*}
We know that the push-forward measure $\lambdarc=\operatorname{F}_{\#} \mu$ is a finite Borel measure on $[0, 1/2]$. Applying Rokhlin's Disintegration Theorem \cite[Theorem 5.3.1]{AmbrosioGigliSavare2008-GradientFlows}
we deduce that for each $\theta \in [0, 1/2]$ there exists a finite measure $\nu_\theta$ on $\mathbb{S}^{d-1}$ such that $\theta \mapsto \nu_\theta$ is a $\lambdarc$-measurable map, and $\nu_\theta(\mathbb{S}^{d-1})=1$ for $\lambdarc$-a.e. $\theta\in [0, 1/2]$. Moreover, we have
\begin{equation}\label{eq:disintegration identity}
    \int_{\Sd} f(\bm x)\, \d\mu(\bm x)= \int_{0}^{1/2} \left(\int_{\mathbb{S}^{d-1}} f(\x'\sin(2\pi \theta), \cos(2\pi \theta))\,\d\nu_\theta(\bm x')\right)\,\d\lambdarc(\theta),
\end{equation}
for all $f \in L^1(\Sd, \d\mu)$. Since $\mu$ is a polar measure, we know the left-hand side of \eqref{eq:disintegration identity} is preserved under the action of $R\in \mathrm{SO}_{\bm v}(d+1)$. On the other hand, this action also preserves the polar angle $\theta\in [0, 1/2]$, thus it must preserve the inner integral on the right-hand side of \eqref{eq:disintegration identity} for $\lambdarc$-a.e. $\theta\in[0, 1/2]$.
Noting the fact that $\nu_{\theta}(\mathbb{S}^{d-1})=1$ and applying the uniqueness of the uniformly distributed measures on metric spaces \cite[Theorem 3.4]{MattilaGeometryOfSets1995}, we deduce
that $\nu_\theta=\sigma_{d-1}$ for $\lambdarc$-a.e. $\theta\in [0, 1/2]$. Now let $A\subseteq [0, 1/2]$ be a Borel set and consider the indicator function of the set $A^d$ defined in \eqref{eq:liftofasetA}. With this choice in \eqref{eq:disintegration identity}, we see
\begin{equation*}
    \mu(A^d)=\int_{0}^{1/2} 1_{A}(\theta)\, \d\lambdarc(\theta)=\lambdarc(A).
\end{equation*}
On the other hand,
\begin{equation*}
\begin{split}
\muarc_p(A)=\mu_p(A\cap(0, 1/2))+\frac{1}{2}\mu_p(A\cap\{0, 1/2\})=\mu(A^{+, d})+\mu(A^d\cap\{\v, -\v\})=\mu(A^d)=\lambdarc(A)
\end{split}
\end{equation*}
Hence, we conclude $\lambdarc=\muarc_p$ on $[0, 1/2]$. Moreover, we have the identity
\begin{equation*}
\begin{split}
    \int_{\Sd} f(\x)\,\d\mu(\x)&=\int_{0}^{1/2} \left(\int_{\mathbb{S}^{d-1}} f(\x'\sin(2\pi \theta), \cos(2\pi \theta))\,\d\sigma_{d-1}(\bm x')\right)\,\d\muarc_p(\theta)\\
    &=\frac{1}{2}\int_{\T} \left(\int_{\mathbb{S}^{d-1}} f(\x'\sin(2\pi \theta), \cos(2\pi \theta))\,\d\sigma_{d-1}(\bm x')\right)\,\d\mu_p(\theta).\qedhere
\end{split}
\end{equation*}
\end{proof}

\subsection{Setup} Let $\mu$ be a polar measure on $\Sd$ and $N\geq1$ be an integer.
We say a function $f:\Sd\to \R$ is a \emph{polar function (with the pole $\v$)} if 
\begin{equation*}
    f(R\x)=f(\x)
\end{equation*}
holds for all $\x\in \Sd$ and for all $R\in \mathrm{SO}_{\v}(d+1)$. 
We first show that an appropriate symmetrization process, considered in \cite{LiTrigonometricExtremalFunctionsErdosTuran1999}, does not increase the radius of the last sign change. Thus, for the problem in higher dimensions, it suffices to consider only the polar functions. 

\medskip

Following the approach of \cite{LiTrigonometricExtremalFunctionsErdosTuran1999} for any function $f:\Sd\to \C$, we define its symmetrization $\widetilde{f}:\Sd\to \C$ as follows
\begin{equation*}
    \widetilde{f}(\bm{x})\coloneqq \int_{\mathrm{SO}_{\bm{v}}(d+1)} f(R\bm{x})\, \d\eta(R),
\end{equation*}
where $\eta$ is the normalized Haar measure on $\mathrm{SO}_{\bm{v}}(d+1)$ so that $\eta(\mathrm{SO}_{\bm{v}}(d+1))=1$. 
First, we claim that this symmetrization process does not increase the radius of the last sign change. 
To justify this, we recall the following facts about spherical harmonics from \cite[Chapter IV]{SteinWeissFourierAnalysis}.
\begin{lemma}[cf. {\cite{SteinWeissFourierAnalysis}}]\label{lem:facts-about-zonal-spherical-harmonics}
    Let $k\geq 0$ and $Y \in \mathscr{H}_k$ be a spherical harmonic on $\Sd$ with $d\geq 2$. The following holds
    \begin{enumerate}[label=\textnormal{(\roman*)}]
        \item The symmetrization of $Y$ is given by $\wt Y(\x)=\frac{\omega_d}{\dim{\mathscr{H}_k}}Y(\v) \,Z_{\v}^{(k)}(\x)$, where $Z^{(k)}_{\v}\in \mathscr{H}_k$ is the zonal spherical harmonic of degree $k$ with pole $\v$.
        \item There exists a polynomial $P_k$ of degree $k$ and a constant $c_{k, d}$ such that
        \begin{equation*}
            Z^{(k)}_{\v}(\x)=c_{k, d}\, P_k(\x\cdot \v), \quad \text{ for all } \x\in\Sd.
        \end{equation*}
    \end{enumerate}
\end{lemma}
\begin{proof}
    For part (i) see \cite[Ch IV, Theorem 2.12]{SteinWeissFourierAnalysis} and \cite[Ch IV, Corollary 2.9]{SteinWeissFourierAnalysis}. \\
    For part (ii) see \cite[Ch IV, Theorem 2.14]{SteinWeissFourierAnalysis}.
\end{proof}

\begin{lemma}\label{lem:symmetrization-of-enn-functions}
    If $0\not\equiv f\in \mathcal{T}(d, \mu, N)$ is a non-trivial function, then $\wt f\in \mathcal{T}(d, \mu, N)$. Moreover, $$\r(\wt f)\leq \r(f).$$
\end{lemma}
\begin{proof}
Since $f(\x)=\sum_{k=0}^{N} Y_k(\x)$, using part (i) of Lemma \ref{lem:facts-about-zonal-spherical-harmonics}, we know $\widetilde{f}$ is a linear combination of zonal spherical harmonics of degree at most $N$. From the definition of $\widetilde{f}$ it is clear that $\widetilde{f}(\x)\geq 0$, for all $\x \not\in B_{\r(f)}(\v)$, i.e., $\r(\wt f)\leq \r(f)$. Moreover, we have 
\begin{equation*}
\begin{split}
    \int_{\Sd} \widetilde{f}(\x) \,  \d\mu(\x)&
    = \int_{\mathrm{SO}_{\bm{v}}(d+1)} \int_{\Sd}  f(R\bm{x}) \,\d\mu(\x) \, \d\eta(R)\\
    &= \int_{\Sd}  f(\bm{x}) \int_{\mathrm{SO}_{\bm{v}}(d+1)}   \,\d\eta(R) \,\d\mu(\x)=\int_{\Sd}  f(\bm{x}) \,\d\mu(\x)\leq 0.
\end{split}
\end{equation*}
Therefore, we can conclude $\wt f\in \mathcal{T}(d, \mu, N)$.    
\end{proof}

\begin{remark}\label{rem:symmetrization not identically zero}
Note that, for a function $f\in \mathcal{T}(d, \mu, N)$ with $\r(f)<\pi$ and $\wt f\equiv 0$, we must have $f\equiv 0$.  Indeed, $\wt f\equiv 0$ together with the fact that $\r(\wt f)<\pi$ implies that $f(\x)=0$ for all points in a small spherical cap centered at $-\v$. Composing $f$ with the stereographic projection, we see that this is a rational function in $\R^d$ vanishing on an open ball, thus $f\equiv 0$. 
\end{remark}
Hence, we can conclude that 
\begin{equation*}
    \rho(d, \mu, N)=\inf_{\substack{0\not \equiv f\in \mathcal{T}(d, \mu, N)\\ f \text{ is polar }}} \r(f).
\end{equation*}
In the rest of the section, by making explicit constructions, we prove the dimension shift result.  
\subsection{Proof of Theorem \ref{thm:reduction-to-dimension-one}} Recall $\mu_p$ denotes the polar part of the measure $\mu$.

\subsubsection{Upper bound}
Let $g\in \mathcal{T}(\mu_p, N)$ be a trigonometric polynomial that is not identically zero. As noted in the remark after the problem \ref{prob:trig-polynomials}, the even part of the function has a smaller radius of last sign change. Hence, we can assume $g$ is even. Consider the function $f:\Sd\to \R$ given by
\begin{equation*}
    f(\x)=g\left(\frac{1}{2\pi}\arccos{(\x\cdot\v)}\right). 
\end{equation*}
We claim that $f\in \mathcal{T}(d, \mu, N)$. First of all, note that $g$ being an even trigonometric polynomial, it only contains cosine terms, thus it can be written as $g(\theta)=P(\cos(2\pi \theta))$, for some polynomial  $P$ of degree at most $N$, and we have  
$f(\x)=P(\x\cdot\v)$. Also note that $f$ is a polar function by definition. Hence, by Lemma \ref{lem:facts-about-zonal-spherical-harmonics}, the spherical harmonics expansion of $f$ can only contain zonal spherical harmonics of degree at most $N$ with the pole $\v$. It is clear that $f(\x)\geq 0$ for all $\x \in \Sd$ with $\x\cdot\v\leq \cos(2\pi \r(g))$, i.e., $\r(f)\leq 2\pi \r(g)$\footnotemark.
\footnotetext{Note that on each side we are using different metrics to define $\r(f)$ and $\r(g)$.}
By Proposition \ref{prop:decomposition-for-polar-measures} we know
\begin{equation*}
    \int_{\Sd} f(\x)\,\d\mu(\x)=\int_{0}^{1/2} f\bigl(\x^\prime \sin(2\pi \theta), \cos(2\pi \theta) \bigr) \,\d\muarc_p(\theta)= \frac{1}{2}\int_{\T} g(\theta)\,\d\mu_p(\theta)\leq 0.
\end{equation*}
Hence, $f\in \mathcal{T}(d, \mu, N)$ and by the choice of $g$ we have $f\not\equiv
0$. Thus
\begin{equation*}
    \rho(d, \mu, N)\leq 2\pi \,\rho(\mu_p, N).
\end{equation*}

\subsubsection{Lower bound}

Without loss of generality, we can assume that there exists a non-trivial $f\in \mathcal{T}(d, \mu, N)$ such that $\r(f)<\pi$, otherwise the upper bound mentioned above and the fact that $\rho(\mu_p, N)\leq 1/2$ completes the proof. Fix any such function $f$, with $\r(f)<\pi$. We have observed in Lemma \ref{lem:symmetrization-of-enn-functions} that $\r(\wt f)\leq \r(f)$, hence we can work with $\wt f$. By definition, $\wt f$ is a polar function, and by Lemma \ref{lem:facts-about-zonal-spherical-harmonics}, we know its development in spherical harmonics has the following form:
\begin{equation*}
    \wt f(\x)=\sum_{k=0}^{N} a_k Z_{\v}^{(k)}(\x)=\sum_{k=0}^{N} a_k \,c_{k, d}\, P_k(\x\cdot\v).
\end{equation*}
Hence, we see that $\wt f$ is the lift of the function $g:\T\to \R$ in the sense $\wt f(\x)=g(\arccos(\x\cdot\v)/2\pi)$, where
\begin{equation*}
    g(\theta)=\sum_{k=0}^{N} a_k\, c_{k, d}\,  P_k\bigl(\cos(2\pi \theta)\bigr).
\end{equation*}
Clearly $g$ is a trigonometric polynomial of degree at most $N$. Also note that by the remark below Lemma \ref{lem:symmetrization-of-enn-functions}, we know $g\not \equiv 0$. Moreover, we have $g(\theta)\geq 0$, for all $\r(\wt f)/2\pi \leq \lvert \theta\rvert\leq 1/2$ and by Proposition \ref{prop:decomposition-for-polar-measures} we have
\begin{equation*}
    0\geq \int_{\Sd} \wt{f}(\x) \, \d\mu(\x)=\int_{\Sd} g(\arccos{(\x \cdot \v)}/2\pi) \, \d\mu(\x) = \frac{1}{2}\int_{\T} g(\theta) \, \d\mu_p(\theta).
\end{equation*}
Thus we conclude  $g\in \mathcal{T}(\mu_p, N)$ and $2\pi \r(g)\leq \r(\wt f)\leq \r(f)$. Therefore, $2\pi \,\rho(\mu_p, N) \leq \rho(d, \mu, N)$.
Combining these two inequalities, we conclude the assertion of Theorem \ref{thm:reduction-to-dimension-one}.

\section{Proof of Theorem \ref{thm:best-radius-of-negativity}}

In this section, we present the proof of our result for trigonometric polynomials on the unit circle. Before giving the proof, we need to recall the machinery of reproducing kernel Hilbert spaces and orthogonal polynomials. In the first part of this section, we will collect relevant facts from \cite{LiReproducingKernelHilbertSpaces1997, LiTrigonometricExtremalFunctionsErdosTuran1999}.  

\subsection{Reproducing kernel Hilbert spaces of polynomials}\label{subsec: rkhs of polynomials}

For each $n\geq 0$, denote by $\mathcal{P}_n$ the set of polynomials of degree at most $n$ with complex coefficients. Recall that, for $P\in \mathcal{P}_n$, we denote by $P^{*, m}$ the reversed polynomial given by  
\begin{equation*}
    P^{*, m}(z)=z^{m}\overline{P(1/\overline{z})}
\end{equation*}
and the shortcut notation $P^{*}=P^{*, \deg(P)}$.

\subsubsection{Setup} Let $P(z)$ be a polynomial of degree $n+1$ such that
\begin{equation}\label{eq:hermite-biehler-inequality}
|P^*(z)|< |P(z)|,
\end{equation}
holds for all $z\in\mathbb{D}$. 

\begin{remark}\label{rem:hermite-biehler-condition-for-zeros}
    Note that the condition \eqref{eq:hermite-biehler-inequality} is equivalent to the fact that $P(z)$ has no zeros in $\mathbb{D}$, and it has at least one zero in $\C\setminus \overline{\mathbb{D}}$.
\end{remark}
  For our purposes, we may assume that $P$ has no zeros on the unit circle. However, this assumption can be dropped with a slight adjustment in several parts of the arguments presented here. Consider the Hilbert space $\mathcal{H}_n(P)$ that consists of polynomials in $\mathcal{P}_n$ and equipped with the inner product
\begin{equation*}
    \langle Q, R \rangle_{\mathcal{H}_n(P)}\coloneqq \int_{\T} Q(z) \overline{R(z)} |P(z)|^{-2}\, \d t,
\end{equation*}
where $z=e^{2\pi i t}$. It is a well-known fact that the space $\mathcal{H}_n(P)$ is a reproducing kernel Hilbert space with this inner product and the function
\begin{equation}\label{eq:reproducing-kernel}
    K(w, z)=\frac{P(z)\overline{P(w)}-P^*(z)\overline{P^*(w)}}{1-\overline{w}z}, \quad \overline{w}z\neq 1,
\end{equation}
is the reproducing kernel\footnote{The values $K(z, z)$, for $z\in \dD$, are defined by passing to the limit $w\to z$.} of the space; see for instance \cite{LiHamiltonRiemannHypothesisPolynomialsOrthogonal1994}.  Thus, it verifies the identity
\begin{equation*}
    \langle Q, K(w, \cdot)\rangle_{\mathcal{H}_n(P)}=Q(w),
\end{equation*}
for all $w\in \C$ and $Q\in \mathcal{H}_n(P)$. Since $K(w, w)=\langle K(w, \cdot), K(w, \cdot)\rangle_{\mathcal{H}_n(P)}$, we know $K(w, w)=0$ implies that $K(w, z)=0$ for all $z\in \mathbb{C}$, by the expression of $K(w, z)$ we deduce $P(w)=P^*(w)=0$ must hold. However, using the estimate \eqref{eq:hermite-biehler-inequality} we see that $w\in \dD$ and by our assumption $P$ has no zeros on $\dD$, a contradiction. Thus, $K(w, w)>0$ for all $w\in \C$.

\medskip

Consider the functions $A(z)$ and $B(z)$ given by 
\begin{align*}
    A(z)=\frac{1}{2} \bigl(P(z)+P^*(z)\bigr) \quad \text{ and } \quad 
    B(z)=\frac{i}{2} \bigl(P(z)-P^*(z)\bigr).
\end{align*}
Observe that the estimate \eqref{eq:hermite-biehler-inequality} implies $\operatorname{deg}(A)=\operatorname{deg}(B)=n+1$ and all of their zeros lie on the unit circle. Moreover, we know $A^*(z)=A(z)$ and $B^*(z)=B(z)$ and $P(z)= A(z)-i B(z)$.  The reproducing kernel of the space $\mathcal{H}_n(P)$ can be represented in terms of $A$ and $B$ as follows:
\begin{equation}\label{eq:reproducing-kernel-with-A-and-B}
    K(w, z)=\frac{2}{i}\frac{B(z)\overline{A(w)}-A(z)\overline{B(w)}}{1-\overline{w}z}, \quad \overline{w}z\neq 1.
\end{equation}
Since the inner product of any of the two becomes $K(\xi_1, \xi_2)$, where $A(\xi_1)=A(\xi_2)=0$, it means $K(\xi_1, z)$ and $K(\xi_2, z)$ are orthogonal. From this, we see that the set of polynomials $\{K(\xi, z): A(\xi)=0\}$ form an orthogonal basis for $\mathcal{H}_n(P)$. 

\smallskip

Considering $e^{i\theta} P(z)$ instead of $P(z)$ in all the places above, we obtain the same Hilbert space $\mathcal{H}_n(P)$. As a general fact we know for every $\theta\in \R$ the family of polynomials given by $\{K(\xi, z) : T_\theta(\xi)=0\}$ forms an orthogonal basis for $\mathcal{H}_n(P)$, where
\begin{equation*}
    T_\theta(z)\coloneqq e^{i\theta} P(z)-e^{-i\theta} P^*(z).    
\end{equation*}
Using this basis and orthogonality, one can obtain the following identity
\begin{equation}\label{eq:quadrature formula for the norm in RKHSoP}
    \lVert Q\rVert_{\mathcal{H}_n(P)}^2=\sum_{T_\theta(\xi)=0} \frac{\lvert Q(\xi)\rvert^2}{K(\xi, \xi)},
\end{equation}
for all polynomials $Q\in \mathcal{H}_n(P)$. This is a polynomial analogue of Theorem 22 of de Branges \cite{deBranges1968-HilbertSpacesOfEntireFunctions}, which was proved in \cite{LiReproducingKernelHilbertSpaces1997}

\subsubsection{The phase function}\label{subsec:phase-function}
Since the polynomial $P(z)$ has no zeros on the unit circle $\partial\mathbb{D}$, and the inequality \eqref{eq:hermite-biehler-inequality} holds, we are allowed to define an analytic branch of $\log P(z)$ in a small neighborhood of the closed unit disk $\overline{\mathbb{D}}$. Define the \emph{phase function} by 
\begin{equation*}
    \varphi(t)\coloneqq  (n+1) \pi t -\operatorname{arg} P(e^{2\pi i  t}),
\end{equation*}
for all $t\in \R$. Note that $\varphi(t)$ is a well-defined continuous function. Since the argument term is a 1-periodic function, we deduce that $\varphi(t+1)-\varphi(t)=\pi (n+1)$.
\begin{lemma}\label{lem:derivative-of-the-phase}
    For all $t\in \R$ we have
    \begin{equation*}
        \varphi^\prime(t)=\pi \cdot \frac{K(e^{2\pi i t}, e^{2\pi i t})}{|P(e^{2\pi i t})|^2}>0.
    \end{equation*}
\end{lemma}

\begin{proof}
Recall that $K(w, z)$ is given by \eqref{eq:reproducing-kernel} for all $w\neq z$. Passing to the limit $w\to z\in \dD$ we see that 
\begin{equation*}
    K(z, z)=\frac{P(z)\overline{P^\prime(z)}-P^\ast(z)\overline{(P^*)^\prime(z)}}{-z}.
\end{equation*}
Hence, using the identity $P^*(z)=z^{n+1}\overline{P(1/\overline{z})}$ we deduce
\begin{equation*}
        K(e^{2\pi i t}, e^{2\pi i t})=(n+1) |P( e^{2\pi i t})|^2-2\operatorname{Re}\left[ e^{2\pi i t} P^\prime( e^{2\pi i t}) \overline{P( e^{2\pi i t})}\right].
\end{equation*}
Using the identity $\arg P(z)=\operatorname{Im} \log P(z)$ and the fact that $\log P(z)$ is analytic in the neighborhood of the unit circle, we conclude the assertion.
\end{proof}

In the following lemma, we present some properties of zeros of the polynomials $T_\theta$. 

\begin{lemma}\label{lem:interlacing-property-of-zeros}
    For each  real number $\theta$ define the polynomial  $T_\theta(z)\coloneqq e^{i\theta} P(z)- e^{-i\theta}P^*(z)$. The following holds:
    \begin{enumerate}[label=\textnormal{(\roman*)}]
        \item $\deg(T_\theta)=n+1$ and all zeros of $T_\theta$ lie on the unit circle;
        \item For a point  $\xi=e^{2\pi i t}$ to be a zero of $T_\theta(z)$ it is necessary and sufficient that $t$ satisfy the equation
        \begin{equation}\label{eq:modular-equation-for-the-phase}
            \varphi(t)\equiv \theta \pmod{\pi}.
        \end{equation}
        As a consequence, we know that all zeros of $T_\theta$ are simple.
        \item If $\theta_1-\theta_2\not \equiv 0 \pmod{\pi}$, then the zeros of $T_{\theta_1}$ and $T_{\theta_2}$ interlace on the unit circle.
    \end{enumerate}
\end{lemma}
\begin{proof}
    (i). First of all, note that the leading coefficient of $T_\theta$ is given by $e^{i\theta} \overline{P^*(0)}-e^{-i\theta} \overline{P(0)}$. Due to the estimate \eqref{eq:hermite-biehler-inequality}, this leading coefficient never vanishes, hence $\deg (T_\theta)=n+1$. By the inequality \eqref{eq:hermite-biehler-inequality} we deduce that $T_\theta$ has no zeros in the unit disk $\mathbb{D}$. Similarly, the reversed inequality holds in $\mathbb{C}\setminus \overline{\mathbb{D}}$ and it implies that $T_\theta$ has no zeros in this set. Thus, the proof of part (i) is complete.

    (ii) For a point $\xi=e^{2\pi i t}$ to satisfy $T_\theta(\xi)=0$ it is necessary and sufficient that 
    \begin{equation*}
       \arg\{e^{i\theta} P(\xi)\}- \arg\{e^{-i\theta}P^*(\xi)\}\equiv0 \pmod{2\pi}.
    \end{equation*}
    Equivalently, 
    \begin{equation*}
        2\theta+\arg\{P(\xi)\}-\arg\{\xi^{n+1}\overline{P(\xi)}\}\equiv 2\theta -2\varphi(t)\equiv 0\pmod{2\pi}.
    \end{equation*}
    Hence, it shows that all the zeros of $T_\theta$ can be obtained by the solutions of the equation \eqref{eq:modular-equation-for-the-phase}. In particular, by Lemma \ref{lem:derivative-of-the-phase} we deduce that all the zeros of $T_\theta$ are simple.  Since $\varphi(t)\pmod{\pi}$ is a 1-periodic function, we can select the points $t_j\in \T$ such that the set of zeros of $T_\theta$ is given by $\{\xi_j=e^{2\pi i t_j} \;|\; j=1,2, \ldots, (n+1)\}$. 

    (iii) Let $t_1<t_2<\ldots <t_{n+1}<t_1+1$ be a set of real numbers such that each $\xi_j=e^{2\pi i t_j}, j=1, 2, ..., n+1$ is a zero of $T_\theta(z)$. 
    By Lemma \ref{lem:derivative-of-the-phase} we know that 
    \begin{equation*}
        \varphi(t_1)<\varphi(t_2)<\ldots<\varphi(t_{n+1})<\varphi(t_1+1)=\pi (n+1) +\varphi(t_1). 
    \end{equation*}
    By part (ii), we know $\varphi(t_j)=\theta+\pi m_j$ for some $m_j\in\Z$. Combining these two facts we deduce that $m_{j+1}-m_j= 1$ and it implies that $\varphi(t_{j+1})-\varphi(t_j)=\pi$. Now, we prove the assertion (iii) by contradiction. Suppose that there are at least two zeros of a polynomial $T_{\theta_1}$ that lie between two zeros of $T_{\theta_2}$. Without loss of generality, we can assume that these zeros are given by 
    \begin{equation*}
        e^{2\pi i t_1}, e^{2\pi i s_1}, e^{2\pi i s_2}, e^{2\pi i t_2},
    \end{equation*}
    where $t_1<s_1<s_2<t_2<t_3<\ldots <t_{n+1}<t_1+1$. By the previous observation and the fact that $\varphi$ is a strictly increasing function, we obtain that 
    \begin{equation*}
        \pi=\varphi(t_2)-\varphi(t_1)>\varphi(s_2)-\varphi(s_1)=\pi.
    \end{equation*}
    Clearly, it is absurd. Hence, we obtain the desired conclusion.
\end{proof}

\subsection{Orthogonal polynomials on the unit circle}\label{subsec: opuc in detail}
In this subsection, we will give a brief overview of the theory of orthogonal polynomials on the unit circle. A more detailed treatment on the subject can be found in \cite{SimonPart1OrthogonalPolynomialsUnitCircle2005}. 

\subsubsection{OPUC for non-trivial measures}
Recall from $\mathsection \ref{subsec: opuc basics}$ that we are working with probability measures on $\dD$ and for a non-trivial probability measure $\mu$ consider the Hilbert space $L^2(\dD, \d\mu)$ with the inner product given by 
\begin{equation*}
    \langle f, g\rangle\coloneqq \int_{\dD} f(z)\overline{g(z)}\,\d\mu(z).
\end{equation*}
Let the sequence $\Phi_n(z)=\Phi_n(z; \d\mu)$ be the sequence of monic orthogonal polynomials in $L^2(\dD, \d\mu)$ and let $\varphi_n(z)\coloneqq \Phi_n(z)/\lVert \Phi_n\rVert_2$ be the sequence of orthonormal polynomials. These polynomials have many interesting properties, including a recursion formula. It is known that  $\Phi_n$ satisfies 
\begin{equation}\label{eq:szego recursion}
    \Phi_{n+1}(z)=z \Phi_{n}(z) -\overline{\alpha_n}\Phi_{n}^{\ast}(z), \quad n=0, 1, \ldots, 
\end{equation}
where $\alpha_n=\alpha_n(\mu) \in \mathbb{C}$ are called \emph{Verblunsky coefficients of $\mu$} and $\Phi_n^\ast=\Phi_n^{\ast,n}$ is the reversed polynomial of $\Phi_n$. One can write the recursion \eqref{eq:szego recursion} in terms of reversed polynomials as follows: 
\begin{equation}\label{eq:szego recursion for reversed polynomials}
        \Phi_{n+1}^*(z)=\Phi_{n}^*(z)-\alpha_{n} z \Phi_{n}(z).
\end{equation}
We see that \eqref{eq:szego recursion} implies $\lvert \alpha_n \rvert<1$. Indeed, we have $\langle z\Phi_n(z), z\Phi_n(z)\rangle=\langle \Phi_n(z), \Phi_n(z)\rangle$ and 
\begin{equation*}
\lVert \Phi_n(z)\rVert^2=\lVert z\Phi_n(z)\rVert^2=\lVert \Phi_{n+1}(z)+\overline{\alpha_n}\Phi_n^*(z)\rVert^2=  \lVert \Phi_{n+1}(z)\rVert^2 + \lvert \alpha_n\rvert^2 \lVert \Phi_n(z)\rVert^2.  
\end{equation*}
Moreover, we have the following identity
\begin{equation}\label{eq:norm via verblunsky coefficients}
    \lVert \Phi_{n+1}\rVert^2=(1-\lvert \alpha_n\rvert^2) \lVert \Phi_{n}\rVert^2=\prod_{j=0}^n(1-\lvert \alpha_j\rvert^2).
\end{equation}

\subsubsection{OPUC for trivial measures} When we have a probability measure $\mu$ with $\#\supp(\mu)=m$, then the Gram--Schmidt process terminates at the polynomial $\Phi_m(z;\d\mu)$. Thus, we can still define $\Phi_0, \Phi_1, \ldots, \Phi_m$ and $\varphi_0, \varphi_1, \ldots, \varphi_{m-1}$. All the identities \eqref{eq:szego recursion}, \eqref{eq:szego recursion for reversed polynomials}, and \eqref{eq:norm via verblunsky coefficients} are still valid for these polynomials. Clearly, the Verblunsky coefficients $\alpha_0, \alpha_1, \ldots, \alpha_{m-1}$ are well-defined and $|\alpha_k|<1$ for all $0\leq k\leq m-2$ and $|\alpha_{m-1}|=1$, which matches with the fact that $\lVert \Phi_{m}\rVert_2=0$.

\subsubsection{General facts about OPUC}
The following lemmata collect the relevant results concerning orthogonal polynomials associated with probability measures on the unit circle. 

\begin{lemma}\label{lem:properties of orthogonal polynomials}
    Let $\mu$ be a probability measure on $\dD$.
    \begin{enumerate}[label=\textnormal{(\roman*)}]
        \item If $\d\mu(z)=\d\mu(\overline{z})$ holds,  then $\Phi_n(z)$ has real coefficients.
        \item $\varphi_n(z)$ has all its zeros in $\mathbb{D}$ and $\varphi^*_n(z)$ has all its zeros in $\C\setminus\overline{\mathbb{D}}$ for $n \geq 1$. Therefore, 
        \begin{equation*}
            \lvert \varphi_n(z)\rvert <\lvert \varphi_n^*(z)\rvert \quad \text{ and }\quad \lvert \Phi_n(z)\rvert <\lvert \Phi_n^*(z)\rvert, \quad z\in \mathbb{D}.
        \end{equation*}
        \item If $\varphi_0, \varphi_1, \ldots, \varphi_n$ exist and are unique, then they satisfy the following Christoffel-Darboux identity
        \begin{equation*}
            \sum_{j=0}^{n-1} \varphi_j(z)\overline{\varphi_j(w)}=\frac{\varphi_n^*(z)\overline{\varphi_n^*(w)}-\varphi_n(z)\overline{\varphi_n(w)}}{1-\overline{w}z},
        \end{equation*}
        for all $\overline{w}z\neq 1$. For $\overline{w}z=1$, the right-hand side of this identity is understood as a limit.
    \end{enumerate}
\end{lemma}
\begin{proof}
 For part (i), see \cite[Lemma 25]{CarneiroExtremalFunctionsBrangesEuclidean2014}. For part (ii), see \cite[$\mathsection 1.7$]{SimonPart1OrthogonalPolynomialsUnitCircle2005} and finally for part (iii), see \cite[Theorem 2.2.7]{SimonPart1OrthogonalPolynomialsUnitCircle2005}.
\end{proof}

\begin{remark}\label{rem:properties-of-opuc-for-trivial-measures}
    It is crucial to note that if the measure $\mu$ has exactly $m$ points in its support, then the conclusion of part (i) is valid for $n\leq m$ and part (ii) is valid only for the range $n\leq m-1$. For degree $m$ polynomials, we have $\Phi_m(z)$ has all of its zeros on $\dD$. If we further assume the symmetry property $\d\mu(z)=\d\mu(\overline{z})$, then $\Phi_m^\ast(z)=(-1)^m \Phi_m(z)$, and the orthonormal polynomial $\varphi_m(z)$ is not well-defined.
\end{remark}

\subsection{Quadrature formulas for Laurent polynomials} Consider the space $\Gamma_n$ of Laurent polynomials $W$ of the form 
\begin{equation*}
    W(z)=\sum_{k=-n}^n a_k z^k,
\end{equation*}
where $a_k\in \C$. As a combined application of two of the theories outlined above, we can establish a quadrature formula for a function in $\Gamma_n$. 

\subsubsection{Auxiliary spaces}\label{subsec:auxiliary-spaces} We consider the reproducing kernel Hilbert space of polynomials $\mathcal{H}_n(P)$ for some particular choices of polynomials $P$. 
We set 
\begin{equation}\label{eq:choice-of-polynomial-p-as-orthonormal}
    P(z)=\varphi_{n+1}^*(z;\d\mu),    
\end{equation}
whenever the latter exists and is unique.
If the measure $\mu$ is trivial and $\#\supp(\mu)=n+1$, then we define the polynomial $P(z)$ as follows
\begin{equation}\label{eq:choice-of-polynomial-p}
    P(z)=(1+i)\varphi_{n}^*(z;\d\mu) + i z \varphi_{n}(z;\d\mu).
\end{equation}

\begin{proposition}\label{prop:polynomial-properties-for-trivial-measures}
    Let polynomial $P$ be given by \eqref{eq:choice-of-polynomial-p}. Then we have the following: 
    \begin{enumerate}[label=\textnormal{(\roman*)}]
        \item $\deg(P)=n+1$ and $P^*(z)=(1-i) z \varphi_{n}(z;\d\mu) - i \varphi_{n}^*(z;\d\mu)$;
        \item $\lvert P^*(z)\rvert < \lvert P(z)\rvert $, holds for all $z\in \mathbb{D}$;
        \item $P$ has no zeros on the unit circle $\dD$.
    \end{enumerate}
\end{proposition}
\begin{proof}
    (i). It is clear that the leading coefficient of $P(z)$ is non-zero, hence, $\deg(P)=n+1$. Thus, $P^*(z)=z^{n+1}\overline{P(1/\overline{z})}=(1-i) z \varphi_{n}(z;\d\mu) - i \varphi_{n}^*(z;\d\mu)$.

    (ii). Consider the expression 
    \begin{equation*}
        \lvert P(z)\rvert^2-\lvert P^*(z)\rvert^2=\lvert \varphi_n^*(z)\rvert^2-\lvert z \varphi_n(z)\rvert^2.
    \end{equation*}
    By Lemma \ref{lem:properties of orthogonal polynomials} part (ii) we know 
    \begin{equation*}
        \lvert z \varphi_n(z)\rvert<\lvert \varphi_n(z)\rvert<\lvert \varphi_n^*(z)\rvert
    \end{equation*}
     holds for each $z\in \mathbb{D}$. Therefore, the assertion is true.

     (iii). Note that, on the unit circle, we have $ \lvert z \varphi_n(z)\rvert=\lvert \varphi_n(z)\rvert=\lvert \varphi_n^*(z)\rvert>0$. Therefore, for any point $z\in \dD$ we have 
     \begin{equation*}
         \lvert (1+i)\varphi_{n}^*(z;\d\mu)\rvert = \sqrt{2}\lvert z\varphi_n(z)\rvert>\lvert i z \varphi_n(z)\rvert.
     \end{equation*}
     Thus $P(z)\neq 0$.
\end{proof}

In each case, we can find the polynomial $P$ which has no zeros on $\dD$ and satisfies the condition \eqref{eq:hermite-biehler-inequality}. Moreover, the reproducing kernel of the space $\mathcal{H}_n(P)$ can be written as 
\begin{equation}\label{eq:reproducing-kernel-formula}
    K_n(w, z)=\frac{P(z) \overline{P(w)}-P^*(z) \overline{P^*(w)}}{1-\overline{w}z},
\end{equation}
whenever $\overline{w}z\neq 1$. Using these spaces, we obtain a family of quadrature formulas parametrized by $\theta\in\T$ for each measure $\mu$.
We will prove the following generalization of the identity from \cite[Corollary 26]{CarneiroExtremalFunctionsBrangesEuclidean2014}.
\begin{lemma}\label{lem:quadrature-formula-for-laurent-polynomials}
Let $\mu$ be a probability measure on $\dD$ and $W\in \Gamma_n$ be a Laurent polynomial. Let $P$ be given by \eqref{eq:choice-of-polynomial-p-as-orthonormal} or by \eqref{eq:choice-of-polynomial-p}. For each fixed $\theta\in \R$ we have 
\begin{equation}\label{eq:quadrature formula for W}
    \int_{\dD} W(z)\, \d\mu(z)=\int_{\dD}\frac{W(z)}{|P(z)|^2}\,\d t=\sum_{T_\theta(\xi)=0}\frac{W(\xi)}{K_{n}(\xi, \xi)},
\end{equation}
where $z=e^{2\pi i t}$ and $T_\theta(z)=e^{i\theta}P(z)-e^{-i\theta}P^*(z)$.
\end{lemma}

\begin{proof}
Recall the choice of $P(z)$ from \eqref{eq:choice-of-polynomial-p-as-orthonormal} and \eqref{eq:choice-of-polynomial-p}. Using the representation \eqref{eq:reproducing-kernel-formula} of the reproducing kernel $K_n$ of the space $\mathcal{H}_n(P)$ and the Christoffel-Darboux identity from Lemma \ref{lem:properties of orthogonal polynomials} we obtain that 
\begin{equation}\label{eq:christoffel-darboux-identity-of-K}
    K_n(w, z)=\sum_{j=0}^n \varphi_j(z;\d\mu)\overline{\varphi_j(w; \d\mu)}.
\end{equation}

Fix $\theta\in \R$ and consider the polynomial $T_\theta(z)$ whose zeros are denoted by $\xi_1, \xi_2, \ldots, \xi_{n+1}$. Now consider the set of polynomials given by 
\begin{equation*}
    q_j(z)=\frac{K_n(\xi_j, z)}{\sqrt{K_n(\xi_j, \xi_j)}}, \quad 1\leq j\leq n+1.
\end{equation*} 
Recall the fact that $\mathcal{H}_n(P)=\mathcal{H}_n(e^{i\theta}P)$ and the formula \eqref{eq:reproducing-kernel-with-A-and-B}. Since $\xi_1, \xi_2, \ldots, \xi_{n+1}$ are distinct zeros of the polynomial $T_\theta$ we know the orthogonality of $K_n(\xi_j, \cdot)$ and $K_n(\xi_k, \cdot)$ in $\mathcal{H}_n(P)$.  Thus, we know $K_n(\xi_j, \xi_k)=0$, for all $j\neq k$ and therefore $q_1, q_2, \ldots, q_{n+1}$ is a basis of the space $\mathcal{P}_n$. Also, note that using \eqref{eq:christoffel-darboux-identity-of-K} we get 
\begin{equation*}
    \begin{split}
        \langle q_j, q_k\rangle_{L^2(\dD, \d\mu)}&=\frac{\langle K_n(\xi_j, \cdot), K_n(\xi_k, \cdot)\rangle_{L^2(\dD, \d\mu)}}{\sqrt{K_n(\xi_j, \xi_j)}\sqrt{ K_n(\xi_k, \xi_k)}}\\
        &=\frac{1}{\sqrt{K_n(\xi_j, \xi_j) K_n(\xi_k, \xi_k)}}\sum_{0\leq r, s\leq n}\overline{\varphi_r(\xi_j)}\varphi_s(\xi_k)\; \langle \varphi_r, \varphi_s\rangle_{L^2(\dD, \d\mu)}\\
        &=\frac{K_n(\xi_j, \xi_k)}{\sqrt{K_n(\xi_j, \xi_j) K_n(\xi_k, \xi_k)}}=\langle q_j, q_k\rangle_{\mathcal{H}_n(P)}.
    \end{split}
\end{equation*}
Since the $q_j$'s form a basis of the space $\P_n$, we conclude that the inner products $\langle \cdot , \cdot \rangle_{L^2(\dD, \d\mu)}$  and $\langle \cdot, \cdot \rangle_{\mathcal{H}_n(P)}$ are the same for elements of $\P_n$. 
\medskip

Now we come to the statement of the lemma. Because of the linearity, it suffices to prove the statement for $W_j(z)=z^j$, with $-n\leq j\leq n$.  First, consider the case $0\leq j\leq n$. We have 
\begin{equation*}
    \int_{\dD} z^j\, \d\mu(z)=\langle W_j, 1\rangle_{L^2(\dD, \mu)}=\langle W_j, 1\rangle_{\mathcal{H}_n(P)}=\int_{\dD}\frac{z^j}{|P(z)|^2}\, \d t.
\end{equation*}
Moreover, using the identity $\langle q_j, q_k\rangle=\delta_{j,k}$, we can write $1=\sum_{k=1}^{n+1} \langle 1, q_k\rangle q_k(z)$ and 
\begin{equation*}
    \langle W_j, 1\rangle_{\mathcal{H}_n(P)}=\sum_{k=1}^{n+1}\overline{\langle 1, q_k\rangle}\langle W_j, q_k\rangle=\sum_{k=1}^{n+1}\frac{W_j(\xi_k)}{K_n(\xi_k, \xi_k)},
\end{equation*}
where in the last identity we used the fact that $\langle Q, q_k\rangle=Q(\xi_k)/\sqrt{K_n(\xi_k, \xi_k)}$ for all polynomials $Q\in \P_n=\mathcal{H}_n(P)$.

For $-n\leq j\leq -1$, we write
\begin{equation*}
    \int_{\dD} z^j\, \d\mu(z)=\langle 1, W_{-j}\rangle=\int_{\dD} \frac{\overline{W_{-j}(z)}}{|P(z)|^2}\,\d t=\int_{\dD}\frac{z^j}{|P(z)|^2}\,\d t.
\end{equation*}
Similarly, we can conclude the final identity using the fact that $\overline{\xi_k^{-j}}=\xi_k^{j}$. 
\end{proof}

\begin{remark}
    Note that the contents of the last lemma are also known as $(n+1)$-point quadrature formula, and a different proof can be found in \cite[see $\mathsection 6$ and $\mathsection 7$]{JonesNjastadThronMomentTheoryOnUnitCircle}. There, authors refer to the polynomials $T_\theta$ as \emph{para-orthogonal} polynomials. We chose to present it here through a general theory of Hilbert spaces of polynomials.
\end{remark}

\subsection{Proof of Theorem \ref{thm:best-radius-of-negativity}}
We first prove part (ii), using an explicit construction. Let $\#\supp(\mu)=M\leq N<\infty$ and  let $z_1, z_2, \ldots, z_M\in\dD$ be all points in $\supp(\mu)$. Since we have the monotonicity $\rho(\mu, N+1)\leq \rho(\mu,N)$, it suffices to show that $\rho(\mu, M)=0$.
Consider the trigonometric polynomial given by 
\begin{equation*}
    f(x)=\lvert P(e^{2\pi i x})\rvert^2,
\end{equation*}
where $P(z)=\prod_{j=1}^{M} (z-z_j)$ is a polynomial of degree $M$. Clearly, $\deg(f)=M$ and on the other hand $\r(f)=0$, because $f\geq 0$ everywhere on $\T$. Moreover, 
\begin{equation*}
\int_{\T} f(x)\d\mu(x)=\sum_{j=1}^{M} \mu(\{z_j\}) \lvert P(z_j)\rvert^2=0.    
\end{equation*}
Therefore, $f\in \mathcal{T}(\mu, M)$ and $f\not\equiv 0$, hence, $\rho(\mu, M)=0$. 

\smallskip

Now, we prove part (i). To prove the lower bound, we argue by contradiction. Suppose that there exists $0\not\equiv f\in \mathcal{T}(\mu, N)$ such that $\r(f)<\theta_1$ and 
\begin{equation*}
   f(x)=\sum_{n=-N}^N \widehat{f}(n) \, e^{2\pi i n x}.
\end{equation*}
Since $f$ can be extended to a holomorphic function and $f\in \mathcal{T}(\mu, N)$, there exists a point $s\in (\r(f), \theta_1)$ such that $f(s)>0$. 
Set $P(z)=\varphi_{N+1}^\ast(z; \d\mu)$ or $P(z)=(1+i)\varphi_{N}^*(z;\d\mu) + i z \varphi_{N}(z;\d\mu)$ as discussed in $\mathsection \ref{subsec:auxiliary-spaces}$. Also set $A_{N+1}(z)=(P(z)+P^\ast(z))/2$, and  consider the phase function $\varphi$ defined as in $\mathsection\ref{subsec:phase-function}$. Now, take any $\theta\in \R$ which satisfies
\begin{equation*}
   \theta \equiv \varphi(s)\pmod{\pi}.
\end{equation*}
Using Lemma \ref{lem:interlacing-property-of-zeros}, we deduce
$T_\theta(z)$ has a simple zero at $e^{2\pi i s}$ hence $T_\theta$ and $T_{\pi/2}=2 i A_{N+1}$ have distinct roots, thus $\theta \not\equiv \pi/2 \pmod{\pi}$. By the interlacing property of zeros of $T_\theta$ and $T_{\pi/2}$ we deduce that $T_\theta(z)$ has no zeros in the arc $\{e^{2\pi i x} : -\r(f)\leq x\leq \r(f)\}$. Now consider the  Laurent polynomial associated with the function $f$
\begin{equation*}
   W_f(z)=\sum_{n=-N}^N \widehat{f}(n) z^n.
\end{equation*}
Using $\int_{\T} f(x)\,\d\mu(x)\leq 0$ and the quadrature formula shown in Lemma \ref{lem:quadrature-formula-for-laurent-polynomials} we get the following contradiction
\begin{equation*}
   0\geq \int_{\T}f(x)\, \d\mu(x)=\int_{\dD} W_f(z)\, \d\mu(z)=\sum_{T_\theta(\xi)=0} \frac{W_f(\xi)}{K_{N}(\xi, \xi)}\geq\frac{f(s)}{K_N(e^{2\pi i s}, e^{2\pi i s})}> 0.
\end{equation*}
Therefore, we conclude $\r(f)\geq \theta_1$.

\smallskip

To complete the proof, we introduce the function $f$ defined by
\begin{equation*}
    f(x)=\frac{|A_{N+1}(e^{2\pi i x})|^2}{\cos(2\pi \theta_1)-\cos(2\pi x)}.
\end{equation*}
We claim that $f$ is a real-valued trigonometric polynomial of degree exactly $N$. To see this, note first that the numerator $|A_{N+1}(e^{2\pi i x})|^2$ is a non-negative trigonometric polynomial vanishing at $x = \theta_1$. In either case of the definition
\begin{equation*}
    A_{N+1}(z)=\frac{\varphi_{N+1}(z;\d\mu)+\varphi_{N+1}^*(z;\d\mu)}{2}
    \quad \text{or} \quad 
    A_{N+1}(z)=\frac{z\varphi_{N}(z;\d\mu)+\varphi_{N}^*(z;\d\mu)}{2},
\end{equation*}
$A_{N+1}$ has only real coefficients by Lemma \ref{lem:properties of orthogonal polynomials} (i), and hence the numerator also vanishes at $x = -\theta_1$. Since the denominator $\cos(2\pi\theta_1) - \cos(2\pi x)$ vanishes precisely at $x = \pm\theta_1$, the apparent singularities of $f$ at these points are removable. Therefore, the quotient is everywhere 
well-defined, and $f$ reduces to a real-valued trigonometric polynomial of degree exactly $N$. By construction, $f$ changes sign only at $x=\pm\theta_1$ and vanishes with multiplicity exactly two at every point $|x|\neq\theta_1$ for which $A_{N+1}(e^{2\pi i x})=0$. Applying the quadrature formula \eqref{eq:quadrature formula for W} yields 
$\int_{\T} f(x)\,\d\mu(x)=0$, so that $f\in\mathcal{T}(\mu, N)$ 
with $\r(f)=\theta_1$. This, combined with the previous lower bound, gives $\rho(\mu, N)=\theta_1$.

\smallskip

It remains to establish the uniqueness of the extremizers. Let $g\in\mathcal{T}(\mu, N)$ be a trigonometric polynomial satisfying $\r(g)=\theta_1$. Applying the quadrature formula for the Laurent polynomial $W_g$ gives
\begin{equation*}
    0\geq \int_{\T} g(x)\,\d\mu(x)
    =\sum_{A_{N+1}(\xi)=0}\frac{W_g(\xi)}{K_{N}(\xi, \xi)}\geq 0,
\end{equation*}
so equality holds throughout, which forces $W_g(\xi)=0$ at every zero of $A_{N+1}$. Since $\r(g)=\theta_1$, we have $W_g(e^{2\pi i x})\geq 0$ for all $\theta_1\leq |x|\leq 1/2$. Consequently, $g$ must vanish with even multiplicity at each point $x$ with $|x|\neq\theta_1$ such that $A_{N+1}(e^{2\pi i x})=0$, which implies that $g$ is a scalar multiple of $f$ as given by~\eqref{eq:extremizer-function}. Since $\deg(g)\leq N=\deg(f)$, we conclude that $g(x)=Cf(x)$ for some constant $C>0$. This concludes the proof.

\section{Applications}
\subsection{Relation between problems \ref{prob:trig-polynomials} and \ref{prob:polynomials}}

We start with an observation that establishes the connection between these two problems.
Let the function $\operatorname{T} :[0,1] \to [0,1/2]$, given by 
$$\operatorname{T}(x)=\frac{1}{2\pi}\arccos(1-2x).$$

\begin{proposition}\label{prop:equivalence-with-polynomial-problem}
Let $\mu$ be a measure on $[0,1]$ and we set its push-forward as $\nuarc:=\operatorname{T}_{\#} \mu$. The extremal problems \textnormal{\ref{prob:trig-polynomials}} and \textnormal{\ref{prob:polynomials}} are related with the following identity
\begin{equation*}
  \Omega(\mu,\, N) =\frac{1- \cos\bigl(2\pi\,\rho(\nu, N)\bigr)}{2},
\end{equation*}
where $\nu$ is a symmetric extension of $\nuarc$ to $\T$.
\end{proposition}

\begin{proof}
The equivalence is established by two explicit constructions. If $p(x)$ is a polynomial in $\mathcal{P}(\mu,N)$, then the trigonometric
polynomial $f(t) \coloneqq p\bigl(\frac{1-\cos(2\pi t)}{2}\bigr)$ belongs to $\mathcal{T}(\nu, N)$ and satisfies
$\r(f) = \frac{1}{2\pi}\arccos(1-\r(p))$. Indeed, we have the identity
\begin{equation*}
    \int_{\T} f(t)\,\d\nu(t)=2\int_{0}^{1/2} f(t)\, \d\nuarc(t)=2\int_{0}^{1} p(x)\, \d\mu(x)\leq 0.
\end{equation*}

Conversely, given an even trigonometric polynomial $g \in \mathcal{T}(\nu, N)$, define
\begin{equation*}
    q(x) \coloneqq (g\circ \operatorname{T}) (x)=\sum_{k=0}^{N} a_k \cos \bigl(k\arccos(1-2x)\bigr)=\sum_{k=0}^{N} a_k\, T_{k}(1-2x),
\end{equation*}
where the last equality uses the Chebyshev polynomials $T_{k}$.
Immediately, $q$ is a polynomial of degree at most $N$ and 
$q \in \mathcal{P}(\mu, N)$. Together, these two constructions establish the equivalence between $\rho(\nu, N)$ and $\Omega(\mu, N)$,
and show that the extremizers of both problems are in one-to-one correspondence.
\end{proof}
\subsection{Proof of Corollary \ref{cor:uniform-measures-on-spheres}}Recall that Jacobi polynomial $P_n^{(\alpha,\beta)}(x)$ of degree $n$
with parameters $\alpha, \beta > -1$ is defined by the Rodrigues formula:
\begin{equation}\label{eq:definition-of-jacobi-polynomial}
P_n^{(\alpha,\beta)}(x) = \frac{(-1)^n}{2^n n!} (1-x)^{-\alpha}(1+x)^{-\beta} \frac{\d^n}{\d x^n} \left[(1-x)^{n+\alpha}(1+x)^{n+\beta}\right]
\end{equation}
for $x \in (-1, 1)$. When $\alpha = \beta = 0$, the Jacobi polynomials reduce to the Legendre polynomials, and when $\alpha = \beta$, they become the Gegenbauer (ultraspherical) polynomials. In fact, Jacobi polynomials are the orthogonal polynomials associated with the measure given by 
\begin{equation*}
   \d\omega_{\alpha, \beta}(x)= (1-x)^\alpha (1+x)^\beta \d x, \quad x\in (-1, 1).
\end{equation*}
More on Jacobi polynomials can be found in \cite{SzegoOrthogonalPolynomials1939} and \cite{Ismail2005ClassicalAndQuantum}.

Let  $\nu_{\alpha, \beta}$ be a measure on $[0, 1]$  given by $\d\nu_{\alpha, \beta}= x^\alpha (1-x)^\beta\d x$. Clearly it is related to  $\omega_{\alpha, \beta}$ via the identity 
\begin{equation*}
    \int_{0}^1 f(x)\, \d\nu_{\alpha, \beta}(x)=2^{-(\alpha+\beta+1)}\int_{-1}^1 f((1-y)/2) \, \d\omega_{\alpha, \beta}(y). 
\end{equation*}
From this identity, it is clear that $\Psi_{n}(x; \d\nu_{\alpha, \beta})=C_{\alpha, \beta} P_n^{(\alpha, \beta)}(1-2x)$, for $x\in [0, 1]$. 
Moreover, we know $(1-x)\d\nu_{\alpha, \beta}(x)=\d\nu_{\alpha, \beta+1}(x)$. On the other hand, the polar part $\mu_{d-1}$ of the measure $\sigma_d$ is given by \eqref{eq:polar-part-of-surface-measure} also related with these measures for $\alpha=\beta=(d-2)/2$, i.e., 
\begin{equation*}
    \int_{0}^{1/2} f\bigl(\sin^2(\pi \theta)\bigr) \sin(2\pi \theta)^{d-1}\, \d \theta=\frac{2^{2\alpha}}{\pi}\int_{0}^1 f(x)\, \d\nu_{\alpha, \alpha}(x).
\end{equation*}

Using Proposition \ref{prop:equivalence-with-polynomial-problem} and the result of Theorem \ref{thm:sign uncertainty for polynomials} we conclude that 
\begin{align*}
    \rho(\mu_{d-1}, 2n-1)=\frac{1}{2\pi}\arccos\bigl(1-2\Omega(\nu_{\alpha, \alpha}, 2n-1)\bigr)=\frac{1}{2\pi}\arccos(1-2 x_{n,1})\\
    \intertext{ and }
    \rho(\mu_{d-1}, 2n)=\frac{1}{2\pi}\arccos\bigl(1-2\Omega(\nu_{\alpha, \alpha}, 2n)\bigr)=\frac{1}{2\pi}\arccos(1-2\xi_{n,1})
\end{align*}
where $x_{n,1}$ and $\xi_{n,1}$ are the smallest zeros of orthogonal polynomials of degree $n$ associated with measures $\nu_{\alpha, \alpha}$ and $\nu_{\alpha, \alpha+1}$. Using the identity $\Psi_{n}(x; \d\nu_{\alpha, \beta})=C_{\alpha, \beta} P_n^{(\alpha, \beta)}(1-2x)$ we see that $1-2x_{n, 1}$ is the largest zero of $P_n^{(\alpha, \alpha)}$ and $1-2\xi_{n,1}$ is the largest zero of $P_n^{(\alpha, \alpha+1)}$. With this, we conclude the proof of Corollary \ref{cor:uniform-measures-on-spheres}. The result of Corollary \ref{cor:yudins-problem} is an immediate consequence of this result.

\subsection{Proof of Theorem \ref{thm:equivalence-with-quadrature-problem}}
First observe that, if $N\geq m=\#\supp(\mu)$, then we have seen that $\rho(\mu, N)=0$. Suppose $\xi$ is a set of nodes of a  $(N+1)$-point quadrature formula, thus $\exists \lambda_{j}>0$ for each $j=1, 2, \ldots, N+1$. As a particular trigonometric polynomial we choose $f(x)=\lvert \Phi_m(e^{2\pi i x}; \d\mu)\rvert^2$. Then we must have 
\begin{equation*}
    0=\int_{\T} f(x)\, \d\mu(x)=\sum_{j=1}^{N+1} \lambda_{j} f(\xi_j).
\end{equation*}
Clearly, it forces the function $f$ to vanish at each point $\xi_j$, and it is absurd because $\Phi_m(z; \d\mu)$ is a degree $m$ polynomial which cannot have $N+1$ zeros. Thus, in this case, there is no quadrature formula with $N+1$ nodes, and hence $\Lambda(\mu, N+1)=0$.

\smallskip

Now assume $N<\#\supp(\mu)$, and we know $\rho(\mu, N)>0$. By Lemma \ref{lem:quadrature-formula-for-laurent-polynomials}, we know there exist several quadrature formulas, and we can assume $\Lambda(\mu, N+1)>0$. 
Fix $\epsilon>0$ sufficiently small, let $\xi$ be set of nodes with $N+1$ points and 
\begin{equation*}
    \min\bigl\{|\xi_j| : \xi_j \in \xi \bigr\}>\Lambda(\mu, N+1)-\epsilon,
\end{equation*}
thus there exist weights $\lambda_j>0, 1\leq j\leq N+1$ for which \eqref{eq:n-point-quadrature} holds. Suppose by contradiction that there exists a non-trivial $f \in \mathcal{T}(\mu, N)$ such that
$\r(f) < \min\{|\xi_j| : \xi_j \in \xi\}$.
Since $f$ is non-negative on all nodes of $\xi$ and has non-positive $\mu$-integral, the quadrature identity \eqref{eq:n-point-quadrature} gives
\begin{equation*}
    0 \geq \int_{\T} f(x)\,\d\mu(x)
    = \sum_{j=1}^{N+1} \lambda_j\,f(\xi_j) \geq 0,
\end{equation*}
forcing $f$ to vanish, with even multiplicities, at every point of $\xi$. However, from a classical result on the number of zeros of trigonometric polynomials \cite[Chapter X, Theorem 1.7]{ZygmundTrigonometricSeries} we know $f$ can have at most $2N$ zeros. This contradicts $f \not\equiv 0$, so no such $f$ can exist. Therefore, for every $f\in \mathcal{T}(\mu, N)$ we have 
\begin{equation*}
    \r(f)\geq \min\{|\xi_j| : \xi_j \in \xi\}>\Lambda(\mu, N+1)-\varepsilon.    
\end{equation*}
By taking the limit the infimum over $f$ and letting $\varepsilon\to 0$ we deduce
\begin{equation*}
    \rho(\mu, N) \geq \Lambda(\mu, N+1).
\end{equation*}

\smallskip

For the converse, one needs to construct a specific quadrature formulathath achieves the equality. This was already done in the proof of Theorem \ref{thm:best-radius-of-negativity}. Recall that the zeros of the polynomial $A_{N+1}(z;\d\mu)$ are the quadrature nodes for Laurent polynomials. Let $\theta_1, \theta_2, \ldots, \theta_{N+1}$ be such that $A_{N+1}\bigl(e^{2\pi i \theta_k}\bigr)=0$, for each $1\leq k\leq N+1$.
In this case, we know $\rho(\mu, N)=\theta_1$ and by the identity \eqref{eq:quadrature formula for W} we conclude that the set $\{\theta_1, \ldots, \theta_{N+1}\}$ is a set of nodes of a quadrature. Therefore, $\Lambda(\mu, N+1)\geq \theta_1=\rho(\mu, N)$. It completes the proof.

\section{Solution for Polynomials: Proof of Theorem \ref{thm:sign uncertainty for polynomials}}\label{sec:solution-for-polynomials}
Proposition \ref{prop:equivalence-with-polynomial-problem} yields the solution to the problem \ref{prob:polynomials} by taking the measure to the circle $\T$ and applying Theorem \ref{thm:best-radius-of-negativity}. Here we present an alternative approach that allows us to deduce the solution directly.
\subsection{Reduction to an equivalent problem}

We begin with reducing the family $\mathcal{P}(\mu, N)$ to a more specific family of polynomials in the spirit of a result from \cite{Carneiro2024SignUncertaintyBrangesSpaces}. For that, we will use the following factorization lemma. 

\begin{lemma}[cf.\ {\cite[Proposition 11]{Carneiro2024SignUncertaintyBrangesSpaces}}]\label{lem:polynomial-reduction}
Let $n \geq 1$ and let $0 \leq t_1 < t_2 < \cdots < t_n < 1$ be distinct real numbers. Define the monic polynomial
\begin{equation*}
    P(x) = \prod_{k=1}^{n}(x - t_k).
\end{equation*}
For any $r > t_n$, there exist polynomials $Q$ and $R$, both non-negative on $\R$, such that:
\begin{enumerate}[label=\textnormal{(\roman*)}]
    \item $P(x) = (x - r)\,Q(x) + R(x)$ for all $x \in \R$;
    \smallskip
    \item the degrees of $Q$ and $R$ are given by
    \begin{equation*}
        \deg(Q) =
        \begin{cases}
            n-1, & \text{if } n \text{ is odd,} \\
            n-2, & \text{if } n \text{ is even,}
        \end{cases}
        \qquad
        \deg(R) =
        \begin{cases}
            n-1, & \text{if } n \text{ is odd,} \\
            n,   & \text{if } n \text{ is even.}
        \end{cases}
    \end{equation*}
\end{enumerate}
\end{lemma}

The proof proceeds by induction on $n$ and is a straightforward adaptation of the argument in \cite[Proposition 11]{Carneiro2024SignUncertaintyBrangesSpaces}; we therefore omit it.

\begin{proposition}\label{prop:reduction-to-single-sign-change}
    For any polynomial $p\in\mathcal{P}(\mu, N)$ there exist a polynomial $q\in \mathcal{P}(\mu, N)$ which satisfies the following
    \begin{enumerate}[label=\textnormal{(\roman*)}]
        \item $\r(q)=\r(p)$;
        \item $q(x)$ changes its sign only at $x=\r(q)=\r(p)$;
        \item $\deg(q)=N$.
    \end{enumerate}
\end{proposition}

\begin{proof}
Since $N\geq 1$ we can assume there exists $p\in \mathcal{P}(\mu, N)$ with $\r(p)<1$. Suppose $p$ changes sign at points $0 \leq t_1 < t_2 < \cdots < t_n < \r(p) < 1$, so that
\begin{equation*}
    p(x) = p_0(x)\,(x - \r(p))\,\prod_{k=1}^{n}(x - t_k),
\end{equation*}
where $p_0 \geq 0$ on $[0, 1]$. Applying Lemma \ref{lem:polynomial-reduction} with $r = \r(p)$, we obtain non-negative polynomials $Q$ and $R$ satisfying
\begin{equation*}
    p(x) = p_0(x)\,(x - \r(p))^2\,Q(x) + p_0(x)\,(x - \r(p))\,R(x).
\end{equation*}
Since $p_0 \geq 0$ on $[0,1]$, the polynomial $q(x) = p_0(x)\,(x - \r(p))\,R(x)$ satisfies $\r(q) = \r(p)$, while
\begin{equation*}
    \int_{0}^{1} q(x)\, \d\mu(x) \leq \int_{0}^{1} p(x)\, \d\mu(x) \leq 0.
\end{equation*}
Thus, $q\in \mathcal{P}(\mu, N)$ and $q(x)$ changes sign only at $x=\r(p)$.
If $\deg(q)=N$, then we conclude the result, otherwise we consider the polynomial given by
    \begin{equation*}
        q_1(x)=\left(\frac{1-x}{1-\r(p)}\right)^{N-\deg(q)} (x-\r(p)) p_0(x) R(x).
    \end{equation*}
From this expression it is clear that $\r(q_1)=\r(p)$ and $\int_{0}^{1} q_1(x)\, \d\mu(x)\leq \int_{0}^{1} q(x)\,\d\mu(x)\leq 0$ and $\deg(q_1)=N$.
\end{proof}
\begin{remark}
    When we assume that $\supp(\mu)$ is infinite, the integral of $q_1$ is strictly negative, hence we can set $q_2=q_1-\int_0^1 q_1(x)\,\d\mu(x)$ and $\r(q_2)<\r(q_1)$.
\end{remark}

The main result of this subsection reduces Problem \ref{prob:polynomials} to a simpler minimization problem over non-negative polynomials. The approach is inspired by \cite[$\S 3.2$]{Carneiro2024SignUncertaintyBrangesSpaces}.

\begin{proposition}\label{prop:reduction-to-multiplication-problem}
Let $\mu$ be a probability measure on $[0,1]$ with infinite support and $N\geq 1$. Then
\begin{equation}\label{eq:reduction-to-multiplication-problem}
    \Omega(\mu, N) = \inf_{q \in \mathcal{P}_{N-1}^{+}} \frac{\int_{0}^{1} x\, q(x)\, \d\mu(x)}{\int_{0}^{1} q(x)\, \d\mu(x)},
\end{equation}
where $\mathcal{P}_{N-1}^{+}$ denotes the class of polynomials of degree exactly $(N-1)$ that are non-negative on $[0, 1]$. 
\end{proposition}

\begin{proof}

By Proposition \ref{prop:reduction-to-single-sign-change}, it suffices to consider polynomials $p \in \mathcal{P}(\mu, N)$ with $\deg(p) = N$ and of the form
\begin{equation}\label{eq:single-sign-change-polynomials}
    p(x) = (x - \r(p))\,q(x),
\end{equation}
where $q \in \mathcal{P}_{N-1}^{+}$. For such $p$, the constraint $\int_{0}^{1} p(x)\, \d\mu(x) \leq 0$ yields
\begin{equation*}
    \int_{0}^{1} x\, q(x)\, \d\mu(x) \leq \r(p) \int_{0}^{1} q(x)\, \d\mu(x).
\end{equation*}
Since $\#\supp(\mu)=\infty$, we have $\int_0^1 q(x)\,\d\mu(x)>0$ for each polynomial $q\in\mathcal{P}_{N-1}^+$. Therefore, for each $p\in \mathcal{P}(\mu, N)$ of the form \eqref{eq:single-sign-change-polynomials} we have the inequality
    \begin{equation*}
        \r(p)\geq \frac{\int_{0}^{1} x q(x)\, \d\mu(x)}{\int_{0}^{1}q(x)\, \d\mu(x)}.
    \end{equation*}
    Taking the infimum over all such polynomials $p$, we deduce, 
    \begin{equation*}
        \Omega(\mu, N)\geq \inf_{q\in \mathcal{P}_{N-1}^+} \frac{\int_{0}^{1} x q(x)\, \d\mu(x)}{\int_{0}^{1}q(x)\, \d\mu(x)}.
    \end{equation*}
    
    To prove the converse, let $\kappa$ denote the quantity on the right-hand side of \eqref{eq:reduction-to-multiplication-problem}. We fix $\varepsilon>0$ and take $q\in \mathcal{P}_{N-1}^{+}$ such that 
    \begin{equation*}
    \frac{\int_{0}^{1} x q(x)\, \d\mu(x)}{\int_{0}^{1}q(x)\, \d\mu(x)}\leq \kappa+\varepsilon.
    \end{equation*}
    Hence, the function $p(x)=(x-\kappa-\varepsilon)q(x)$ belongs to $\mathcal{P}(\mu, N)$ and $\r(p)=\kappa+\varepsilon$. Thus, $\Omega(\mu, N)\leq \kappa+\varepsilon$. Letting $\varepsilon\to 0$ we conclude the identity \eqref{eq:reduction-to-multiplication-problem}.
\end{proof}

\begin{remark}
On the right-hand side of \eqref{eq:reduction-to-multiplication-problem}, one can take the infimum over $\mathcal{P}_{\leq N-1}^{+}$ (polynomials of degree at most $N-1$ that are non-negative on $[0,1]$) rather than exactly degree $N-1$, since the two infima coincide. Indeed, for any $q \in \mathcal{P}_{\leq N-1}^{+}$ of degree $k < N-1$, define
\begin{equation*}
    q_{\varepsilon}(x) \coloneqq q(x)\bigl(1 + \varepsilon\, x^{N-1-k}\bigr), \quad \varepsilon > 0.
\end{equation*}
Then $q_{\varepsilon} \in \mathcal{P}_{N-1}^{+}$, and since $q_{\varepsilon} \to q$ uniformly on $[0,1]$ as $\varepsilon \to 0^{+}$, we have
\begin{equation*}
    \frac{\int_{0}^{1} x\,q_{\varepsilon}(x)\,\d\mu(x)}{\int_{0}^{1} q_{\varepsilon}(x)\,\d\mu(x)} \;\longrightarrow\; \frac{\int_{0}^{1} x\,q(x)\,\d\mu(x)}{\int_{0}^{1} q(x)\,\d\mu(x)} \quad \text{as } \varepsilon \to 0^{+}.
\end{equation*}
\end{remark}

\medskip

\subsection{Orthogonal polynomials on the real line}
 Given a finite Borel measure $\mu$ on $[0,1]$ with infinite support, let $1\equiv\Psi_0, \Psi_1, \ldots, \Psi_{n}(x;\d\mu), \ldots$ be the unique monic orthogonal polynomials in $L^2(\R, \d\mu)$. 
 For an orthogonal polynomial $\Psi_n(x; \d\mu)$ we denote its zeros by $x_{n, k}$ for $1\leq k\leq n$. 
 The following lemma collects all the necessary facts about the zeros of orthogonal polynomials.
   
\begin{lemma}[{cf. \cite[$\mathsection$ 3.2 -- 3.4 ]{Ismail2005ClassicalAndQuantum}}]\label{lem:basic-facts-about-orthogonal-polynomials-on-real-line}
    Let $\mu$, $\Psi_n$ and $x_{n, k}$ be as above. 
    For each $n\geq 1$, the following holds
    \begin{enumerate}[label=\textnormal{(\roman*)}]
        \item the polynomials $\Psi_0, \Psi_1, \ldots, \Psi_{n}$ satisfy the Christoffel-Darboux identities:
        \begin{align*}
            \sum_{k=0}^{n-1} \frac{\Psi_{k}(x) \Psi_k(y)}{\lVert \Psi_k\rVert_{2}^2}&=\frac{1}{\lVert \Psi_{n-1}\rVert_{2}^2}\frac{\Psi_{n}(x)\Psi_{n-1}(y)-\Psi_{n}(y)\Psi_{n-1}(x)}{x-y}, \quad x\neq y
        \intertext{and}
        \sum_{k=0}^{n-1} \frac{\Psi_{k}^2(x)}{\lVert \Psi_k\rVert_{2}^2}&=\frac{\Psi_{n}^\prime(x)\Psi_{n-1}(x)-\Psi_{n}(x)\Psi_{n-1}^\prime(x)}{\lVert \Psi_{n-1}\rVert_{2}^2}, \quad x\in \R;
        \end{align*}
        \item the polynomial $\Psi_{n}(x; \d\mu)$ has exactly $n$ distinct zeros in $(0, 1)$;
        \item between two zeros of $\Psi_n(x;\d\mu)$ there is exactly one zero of $\Psi_{n-1}(x;\d\mu)$, i.e., 
        \begin{equation*}
            x_{n, k}<x_{n-1, k}<x_{n, k+1}, \quad \text{ for all } \quad k=1, 2,  \ldots n-1;
        \end{equation*}
        \item the zeros of $\Psi_n(x; \d\mu)$ are nodes of a Gauss quadrature formula, that is, there exist real numbers $\lambda_{n, k}>0$, $1\leq k\leq n$, such that for any polynomial $q\in \mathcal{P}_{2n-1}$ we have 
        \begin{equation*}
            \int_{0}^{1} q(x)\,\d\mu(x)=\sum_{k=1}^{n} \lambda_{n, k}\,  q(x_{n, k}).
        \end{equation*}

    \end{enumerate}
\end{lemma}

Based on these facts, we can prove the first part of Theorem \ref{thm:sign uncertainty for polynomials}. 

\begin{proof}[Proof of Theorem \ref{thm:sign uncertainty for polynomials} part \textnormal{(i)}]

Since $\mu$ has infinite support, $\Psi_n(x;\d\mu)$ is uniquely determined for each $n\geq 1$. Consider the problem $\Omega(\mu, 2n-1)$.
In this case, by Proposition \ref{prop:reduction-to-multiplication-problem} we reduce the problem to finding the infimum over the ratio of two integrals. Then, by Lemma \ref{lem:basic-facts-about-orthogonal-polynomials-on-real-line} part (iv) we can write for every $q\in \mathcal{P}_{2n-2}^+$ the following
\begin{equation}\label{eq:lower-bound-with-quadrature-formula}
   \begin{split}
        \int_0^1 x q(x)\, \d\mu(x)=\sum_{k=1}^{n} & \lambda_{n, k}\, x_{n, k} \, q(x_{n, k})\geq x_{n, 1}\sum_{k=1}^{n} \lambda_{n, k} \, q(x_{n, k})=x_{n, 1} \int_0^1 q(x)\, \d\mu(x).
   \end{split}
\end{equation}
Here, the equality can occur if and only if $q(x_{n, k})=0$ for all $2\leq k\leq n$. Hence, we conclude that $\Omega(\mu, 2n-1)=x_{n, 1}$. Moreover, the equality is attained if and only if the polynomial $p\in \mathcal{P}(\mu, 2n-1)$ is a multiple of 
\begin{equation*}
    \frac{\bigl(\Psi_n(x; \d\mu)\bigr)^2}{x-x_{n, 1}}.
\end{equation*}
It completes the proof of part (i).
\end{proof}

In order to prove the last part of Theorem \ref{thm:sign uncertainty for polynomials}, we will need the following tools. 
\subsubsection{Characterization of non-negative polynomials on \texorpdfstring{$[0,1]$}{[0,1]}}
The following fact is a classical result in the theory of orthogonal polynomials. It can be proved using the Fej\'{e}r-Riesz theorem, and its proof can be found in \cite[$\mathsection 1.21$]{SzegoOrthogonalPolynomials1939}.

\begin{lemma}[Theorem of Luk\'{a}cs]\label{lem:lukacs-theorem}
    Let $q(x)$ be a polynomial of degree $k$. If $q(x)\geq 0$ for all $x\in [0, 1]$, then $q(x)$ can be represented as
    \begin{equation*}
        q(x)=\begin{cases}
            x a(x)^2+(1-x) b(x)^2, &\text{ if $k$ is odd};\\
            c(x)^2+x(1-x) d(x)^2, &\text{ if $k$ is even},
        \end{cases}
    \end{equation*}
    where $a, b, c, d$ are polynomials with real coefficients and the degree of each summand on the right does not exceed $k$.
\end{lemma}

\subsubsection{Comparison of zeros}
Let $\mu, \Psi_n(x;\d\mu)$, and $x_{n,k}$ be as above. We define two measures $\mu_0 $ and $ \mu_1$ as follows:
\begin{align*}
  \d\mu_0(x) := x\,\d\mu(x),  \qquad   \d\mu_1(x) := (1-x)\,\d\mu(x).
\end{align*}
We enumerate the zeros of the associated monic orthogonal polynomials $\Psi_{n}(x;\d\mu_0)$ and $\Psi_{n}(x;\d\mu_1)$ by $\zeta_{n, k}$ and  $\xi_{n, k}$, respectively.

\begin{proposition}\label{prop:comparison-of-zeros}
    Let $\mu, \mu_0$, and $\mu_1$ be as above. 
    Let $x_{n, k}, \zeta_{n,k}$, and $ \xi_{n, k}$ be zeros of corresponding orthogonal polynomials as before. Then for each $1\leq k\leq n$ we have the following
    \begin{equation}\label{eq:chain-of-inequalities-of-zeros}
        \xi_{n,k} < x_{n,k} < \zeta_{n,k}.
    \end{equation}
  
\end{proposition}

\begin{proof}
Define the \emph{ratio function} $R_n$ by
\begin{equation*}
    R_n(x)\coloneqq\frac{\Psi_{n+1}(x; \d\mu)}{\Psi_n(x; \d\mu)}.
\end{equation*}
Note that by Lemma \ref{lem:basic-facts-about-orthogonal-polynomials-on-real-line} part (ii) $R_n(0)$ and $R_n(1)$ are well-defined and finite. First of all, note that the monic orthogonal polynomials associated with $\mu_0$ and $\mu_1$ can be expressed as 
\begin{align*}
    \Psi_n(x; \d\mu_1)=\frac{R_n(1)\Psi_n(x; \d\mu)-\Psi_{n+1}(x; \d\mu)}{1-x},
\intertext{and}
    \Psi_n(x; \d\mu_0)=\frac{\Psi_{n+1}(x; \d\mu)-R_n(0)\Psi_n(x; \d\mu)}{x}.
\end{align*}
Indeed, these are two monic polynomials that are orthogonal to all monomials $z^k$ for $0\leq k\leq n-1$ with respect to the corresponding measures. Using Lemma \ref{lem:basic-facts-about-orthogonal-polynomials-on-real-line} part (iii) we conclude that the polynomials $\Psi_{n+1}(x; \d\mu)$ and 
$\Psi_{n}(x; \d\mu)$ have no common zeros. It means $R_n$ has poles exactly at zeros of $\Psi_n(x;\d\mu)$. We apply Lemma \ref{lem:basic-facts-about-orthogonal-polynomials-on-real-line} part (ii) for measures $\mu_0$ and $\mu_1$ to conclude $\zeta_{n, k}, \xi_{n, k}\in (0, 1)$ for each $1\leq k\leq n$. Thus,  
\begin{align}
    \Psi_n(x; \d\mu_1)=0 \;\Longleftrightarrow \; x\neq 1 \text{ and }R_n(x)=R_n(1),\label{eq:zeros-of-mu-one}
\intertext{and}
    \Psi_n(x; \d\mu_0)=0 \;\Longleftrightarrow \; x\neq 0 \text{ and }R_n(x)=R_n(0).
\end{align}
Using the Christoffel-Darboux identity in Lemma \ref{lem:basic-facts-about-orthogonal-polynomials-on-real-line} we get 
\begin{equation*}
\begin{split}
    R_n^\prime(x)&=\frac{\Psi_{n+1}^\prime(x; \d\mu)\Psi_{n}(x; \d\mu)-\Psi_{n+1}(x;\d\mu)\Psi_{n}^\prime(x;\d\mu)}{\Psi_n^2(x; \d\mu)}\\
    &=\sum_{k=0}^{n} \frac{\lVert \Psi_{n}\rVert_{2}^2}{\lVert \Psi_k\rVert_{2}^2}\;\frac{\Psi_{k}^2(x;\d\mu)}{\Psi_n^2(x; \d\mu)}>0,
\end{split}
\end{equation*}
for all $x\neq x_{n,k}$. Recall from Lemma~\ref{lem:basic-facts-about-orthogonal-polynomials-on-real-line} that $x_{n+1,k} < x_{n,k}$ for each $k = 1,\ldots,n$, and that $R_n$ maps each of the intervals
\begin{equation*}
    I_{n, 0}\coloneqq(-\infty, x_{n, 1}), \;\; I_{n,1}\coloneqq(x_{n, 1}, x_{n, 2}),
    \;\;\ldots,\;\; I_{n,n-1}\coloneqq(x_{n, n-1}, x_{n, n}),\;\;
    I_{n,n}\coloneqq(x_{n, n}, +\infty)
\end{equation*}
bijectively onto $(-\infty, +\infty)$, with $x_{n+1,k} \in I_{n,k-1}$ for each $k = 1,\ldots,n+1$. We use these facts to locate the zeros of both $\Psi_n(x;\,\d\mu_1)$ and $\Psi_n(x;\,\d\mu_0)$ simultaneously.

For $\Psi_n(x;\,\d\mu_1)$, its zeros are the solutions $x \neq 1$
of $R_n(x) = R_n(1)$. Since $x_{n+1,n+1} < 1$ and $R_n' > 0$ on $I_{n,n}$, we have $R_n(1) > R_n(x_{n+1,n+1}) = 0$, so the unique solution of
$R_n(x) = R_n(1)$ on $I_{n,n}$ is $x = 1$ itself, which is excluded. Hence no zero $\xi_{n,j}$ lies in $I_{n,n}$. On each interval $I_{n,k-1}$, for $k = 1,\ldots,n$, bijectivity gives exactly one solution. Since $R_n(x_{n+1,k}) = 0 < R_n(1)$, strict monotonicity ensures that a zero of $\Psi_{n}(x; \d\mu_1)$ should belong to $(x_{n+1,k}, x_{n,k}) \subset I_{n,k-1}$. As $\Psi_n(x;\,\d\mu_1)$ has exactly $n$ zeros by Lemma~\ref{lem:basic-facts-about-orthogonal-polynomials-on-real-line}, we conclude $x_{n+1,k} < \xi_{n,k} < x_{n,k}$ for each $k = 1,\ldots,n$.

The argument for $\Psi_n(x;\,\d\mu_0)$ is analogous: its zeros are
the solutions $x \neq 0$ of $R_n(x) = R_n(0)$.
Since $0 < x_{n+1,1}$ and $R_n' > 0$ on $I_{n,0}$, we have
$R_n(0) < R_n(x_{n+1,1}) = 0$, so the unique solution on $I_{n,0}$
is $x = 0$, which is excluded, hence no zero $\zeta_{n,j}$ lies in $I_{n, 0}$. On each $I_{n,k}$, for $k = 1,\ldots,n$, since $R_n(x_{n+1,k+1}) = 0 > R_n(0)$, the unique solution is forced into $(x_{n,k}, x_{n+1,k+1}) \subset I_{n,k}$, giving $x_{n,k} < \zeta_{n,k}$ for each $k = 1,\ldots,n$.
\end{proof}

After collecting all the necessary results from the theory of orthogonal polynomials, we present the last part of the solution to the problem \ref{prob:polynomials}.

\begin{proof}[Proof of Theorem \ref{thm:sign uncertainty for polynomials} part \textnormal{(ii)}]

Let $n\geq 1$ and consider the problem $\Omega(\mu, 2n)$. Again, using Proposition \ref{prop:reduction-to-multiplication-problem}, we reduce the problem to finding the infimum over $q\in \mathcal{P}_{2n-1}^+$. In addition, we apply Lemma \ref{lem:lukacs-theorem} to further restrict the problem into the following
\begin{equation}\label{eq:extremal-problem-for-even-degrees}
    \Omega(\mu, 2n)=\inf_{a, b\in \mathcal{P}_{n-1}} \frac{\int_{0}^{1} x^2 a(x)^2 +x(1-x) b(x)^2\, \d\mu(x)}{\int_{0}^{1}x a(x)^2 + (1-x) b(x)^2\, \d\mu(x)}. 
\end{equation}
On the other hand, we have the Gauss quadrature formula for measures $\mu_0$ and $\mu_1$, and from Proposition \ref{prop:comparison-of-zeros} we know $\xi_{n, 1}<\zeta_{n,1}$, thus 
\begin{equation}\label{eq:quadrature-formula-for-mu1}
\begin{split}
    \int_{0}^{1} x^2 a(x)^2 +x(1-x) b(x)^2\, \d\mu(x)&=\sum_{k=1}^{n}\lambda_{n, k}(\mu_0) \,\zeta_{n, k}\, a(\zeta_{n, k})^2+\lambda_{n, k}(\mu_1) \,\xi_{n, k}\, b(\xi_{n, k})^2\\
    &\geq \xi_{n, 1} \int_{0}^1 x a(x)^2 + (1-x) b(x)^2\, \d\mu(x).
\end{split}
\end{equation}
Similar to the previous argument, the equality can occur if and only if 
$a(\zeta_{n, k})=0$ for all $1\leq k\leq n$, and $b(\xi_{n, k})=0$ for all $2\leq k\leq n$. Since, $a, b\in \mathcal{P}_{n-1}$ we see that $a\equiv 0$ and $b(x)=C\prod_{k=2}^{n}(x-\xi_{n, k})$. As a result of \eqref{eq:extremal-problem-for-even-degrees} and \eqref{eq:quadrature-formula-for-mu1} we deduce 
\begin{equation*}
    \Omega(\mu, 2n)=\xi_{n, 1}.
\end{equation*}
Moreover, the extremizers of the problem must be of the form 
\begin{equation*}
    p(x)=C(1-x)\frac{\bigl(\Psi_n(x;\d\mu_1)\bigr)^2}{x-\xi_{n,1}},
\end{equation*}
for some $C>0$.
\end{proof}

\begin{remark}
    Note that, if we apply the quadrature formula with higher degree orthogonal polynomials, we lose the equality cases of the inequalities in \eqref{eq:lower-bound-with-quadrature-formula} and \eqref{eq:quadrature-formula-for-mu1}. It is worth noting that, different from the case of trigonometric polynomials, we did not need a continuum family of quadrature formulas. 
\end{remark}

\section*{Acknowledgments}
The author is deeply grateful to Emanuel Carneiro for his valuable suggestions, constant encouragement, and careful reading of the earlier version of this manuscript. Special thanks are also due to Antonio Pedro Ramos for his suggestions and help in improving the exposition of the article. The author also wishes to thank Abdilaziz Komilov and Muhammad Awais for their insightful discussions and helpful remarks throughout the preparation of this work.

\bibliographystyle{abbrv}
\bibliography{ref}

@book{SteinWeissFourierAnalysis,
    AUTHOR = {Stein, Elias M. and Weiss, Guido},
     TITLE = {Introduction to {F}ourier analysis on {E}uclidean spaces},
    SERIES = {Princeton Mathematical Series},
    VOLUME = {No. 32},
 PUBLISHER = {Princeton University Press, Princeton, NJ},
      YEAR = {1971},
     PAGES = {x+297},
   MRCLASS = {42A92 (31B99 32A99 46F99 47G05)},
  MRNUMBER = {304972},
}

@book{ZygmundTrigonometricSeries,
 author = {Zygmund, A.},
 title = {Trigonometric series. {Volumes} {I} {and} {II} combined.},
 PUBLISHER = {Cambridge University Press},
 YEAR = {1988},
 PAGES = {Vol. I. xii+383 pp.; Vol. II. vii+354},
 edition = {2nd},
 fseries = {Cambridge Mathematical Library},
 series = {Camb. Math. Libr.},
 isbn = {0-521-35885-X},
 language = {English},
 zbMATH = {192905},
 Zbl = {0628.42001}
}

@article{Bourgain2010PrincipeDHeisenbergFonctionsPositives,
 title = {Principe d'{{Heisenberg}} et Fonctions Positives},
 author = {Bourgain, Jean and Clozel, Laurent and Kahane, Jean-Pierre},
 fjournal = {Annales de l'Institut Fourier},
 journal = {Ann. Inst. Fourier},
 issn = {0373-0956},
 publisher = {Association des Annales de l{\textquoteright}institut Fourier},
 mrnumber = {2722239},
 volume = {60},
 number = {4},
 pages = {1215--1232},
 year = {2010},
 language = {French},
 doi = {10.5802/aif.2552},
 url = {https://aif.centre-mersenne.org/articles/10.5802/aif.2552/},
}

@article{Goncalves2023NewSignUncertaintyPrinciples,
  title = {New {{Sign Uncertainty Principles}}},
  author = {Gon{\c c}alves, Felipe and Oliveira e Silva, Diogo and Ramos, Jo{\~a}o P. G.},
  year = 2023,
  month = jul,
  fjournal = {Discrete Analysis},
  journal={Discrete Anal.},
  eprint = {2003.10771},
  doi = {10.19086/da.84266},
  url = {http://arxiv.org/abs/2003.10771},
  archiveprefix = {arXiv},
  langid = {english}
}

@misc{Carneiro2024SignUncertaintyBrangesSpaces,
  title = {Sign Uncertainty and de {{Branges}} Spaces},
  author = {Carneiro, Emanuel and Ismoilov, Tolibjon and Ramos, Antonio Pedro},
  year = 2024,
  month = aug,
  number = {arXiv:2408.01186},
  eprint = {2408.01186},
  publisher = {arXiv},
  doi = {10.48550/arXiv.2408.01186},
  url = {http://arxiv.org/abs/2408.01186},
  archiveprefix = {arXiv},
  note={arXiv:2408.01186, to appear in Ann. Inst. Fourier}
}

@article{Babenko1984ExtremalProblemPolynomials,
  title = {An Extremal Problem for Polynomials},
  author = {Babenko, A. G.},
  year = 1984,
  month = mar,
  fjournal = {Mathematical Notes of the Academy of Sciences of the USSR},
  journal={Mat. Zametki},
  volume = {35},
  number = {3},
  pages = {181--186},
  issn = {0001-4346, 1573-8876},
  doi = {10.1007/BF01139914},
  url = {http://link.springer.com/10.1007/BF01139914},
  copyright = {http://www.springer.com/tdm},
  langid = {english}
}

@article{CarneiroExtremalFunctionsBrangesEuclidean2014,
  title = {Extremal Functions in de {{Branges}} and {{Euclidean}} Spaces},
  author = {Carneiro, Emanuel and Littmann, Friedrich},
  year = 2014,
  month = aug,
  journal = {Advances in Mathematics},
  volume = {260},
  pages = {281--349},
  issn = {00018708},
  doi = {10.1016/j.aim.2014.04.007},
  url = {https://linkinghub.elsevier.com/retrieve/pii/S0001870814001467},
  urldate = {2026-01-14},
  langid = {english}
}

@article{LiHamiltonRiemannHypothesisPolynomialsOrthogonal1994,
  title = {The {{Riemann}} Hypothesis for Polynomials Orthogonal on the Unit Circle},
  author = {Li, Xian-Jin},
  year = 1994,
  month = jan,
  journal = {Mathematische Nachrichten},
  volume = {166},
  number = {1},
  pages = {229--258},
  issn = {0025-584X, 1522-2616},
  doi = {10.1002/mana.19941660118},
  url = {https://onlinelibrary.wiley.com/doi/10.1002/mana.19941660118},
  urldate = {2026-01-14},
  copyright = {http://onlinelibrary.wiley.com/termsAndConditions\#vor},
  langid = {english}
}

@article{LiReproducingKernelHilbertSpaces1997,
  title = {On Reproducing Kernel {Hilbert} Spaces of Polynomials},
  author = {Li, Xian-Jin},
  year = 1997,
  month = jan,
  journal = {Mathematische Nachrichten},
  volume = {185},
  number = {1},
  pages = {115--148},
  issn = {0025-584X, 1522-2616},
  doi = {10.1002/mana.3211850110},
  url = {https://onlinelibrary.wiley.com/doi/10.1002/mana.3211850110},
  urldate = {2026-01-14},
  copyright = {http://onlinelibrary.wiley.com/termsAndConditions\#vor},
  langid = {english}
}

@article{LiTrigonometricExtremalFunctionsErdosTuran1999,
  title = {Some Trigonometric Extremal Functions and the {{Erdos-Turan}} Type Inequalities},
  author = {Li, Xian-Jin and Vaaler, D.},
  year = 1999,
  journal = {Indiana University Mathematics Journal},
  volume = {48},
  number = {1},
  pages = {0--0},
  issn = {0022-2518},
  doi = {10.1512/iumj.1999.48.1508},
  url = {http://www.iumj.indiana.edu/IUMJ/fulltext.php?artid=1508&year=1999&volume=48},
  urldate = {2026-01-14},
  langid = {english}
}

@book{SimonPart1OrthogonalPolynomialsUnitCircle2005,
  title = {Orthogonal Polynomials on the Unit Circle},
  author = {Simon, Barry},
  year = 2005,
  series = {Colloquium {{Publications}}},
  volume = {54.1},
  publisher = {American Mathematical Society},
  address = {Providence, R.I},
  doi = {10.1090/coll/054.1},
  url = {http://www.ams.org/coll/054.1},
  isbn = {978-0-8218-3446-6 978-1-4704-3199-0},
  langid = {english}
}

@book{SzegoOrthogonalPolynomials1939,
  title = {Orthogonal Polynomials},
  author = {Szeg{\"o}, G{\'a}bor},
  year = 1939,
  series = {Colloquium {{Publications}}},
  edition = {Online-Ausg},
  number = {23},
  publisher = {American mathematical society},
  address = {New York city},
  isbn = {978-0-8218-1023-1 978-1-4704-3171-6},
  langid = {english}
}

@article{YudinTwoExternalProblemsTrigonometric1996,
  title = {Two External Problems for Trigonometric Polynomials},
  author = {Yudin, V A},
  year = 1996,
  month = dec,
  journal = {Sbornik: Mathematics},
  volume = {187},
  number = {11},
  pages = {1721--1736},
  issn = {1064-5616, 1468-4802},
  doi = {10.1070/SM1996v187n11ABEH000175},
  url = {https://www.mathnet.ru/eng/sm175},
  urldate = {2026-01-14},
  langid = {english}
}

@book{MattilaGeometryOfSets1995,
 author = {Mattila, Pertti},
 title = {Geometry of sets and measures in {Euclidean} spaces. {Fractals} and rectifiability},
 fseries = {Cambridge Studies in Advanced Mathematics},
 series = {Camb. Stud. Adv. Math.},
 volume = {44},
 isbn = {0-521-46576-1},
 year = {1995},
 publisher = {Cambridge: Univ. Press},
 language = {English},
 zbMATH = {739280},
 Zbl = {0819.28004}
}

@article{Szego1921,
 author = {Szeg{\"o}, G.},
 title = {{\"U}ber die {Entwicklung} einer analytischen Funktion nach den Polynomen eines Orthogonalsystems},
 fjournal = {Mathematische Annalen},
 journal = {Math. Ann.},
 issn = {0025-5831},
 volume = {82},
 pages = {188--212},
 year = {1921},
 language = {German},
 doi = {10.1007/BF01498664},
 url = {https://eudml.org/doc/158844},
 zbMATH = {2601419},
 JFM = {48.0378.01}
}

@book{Ismail2005ClassicalAndQuantum,
 author = {Ismail, Mourad E. H.},
 title = {Classical and quantum orthogonal polynomials in one variable},
 fseries = {Encyclopedia of Mathematics and Its Applications},
 series = {Encycl. Math. Appl.},
 issn = {0953-4806},
 volume = {98},
 isbn = {0-521-78201-5},
 year = {2005},
 publisher = {Cambridge: Cambridge University Press},
 language = {English},
 zbMATH = {2228141},
 Zbl = {1082.42016}
}

@incollection{AndreevYudin2003PositiveValues,
 author = {Andreev, N. N. and Yudin, V. A.},
 title = {Positive values of harmonic polynomials},
 booktitle = {Function spaces, approximations, and differential equations. Collected papers dedicated to Oleg Vladimirovich Besov on his 70th birthday. Transl. from the Russian},
 pages = {39--45},
 year = {2003},
 publisher = {Moscow: Maik Nauka/Interperiodika},
 language = {English},
 zbMATH = {2195998},
 Zbl = {1124.41026}
}

@article{Yudin2004Onpositivevaluesofsphericalharmonics,
 author = {Yudin, V. A.},
 title = {On positive values of spherical harmonics and trigonometric polynomials},
 fjournal = {Mathematical Notes},
 journal = {Math. Notes},
 issn = {0001-4346},
 volume = {75},
 number = {3},
 pages = {447--450},
 year = {2004},
 language = {English},
 doi = {10.1023/B:MATN.0000023328.29401.16},
 zbMATH = {2121378},
 Zbl = {1074.33012}
}

@article{JonesNjastadThronMomentTheoryOnUnitCircle, 
author = {Jones, William B. and Nj{\ringaccent{a}}stad, Olav and Thron, W. J.},
title = {Moment Theory, Orthogonal Polynomials, Quadrature, and Continued Fractions Associated with the unit Circle},
fjournal = {Bulletin of the London Mathematical Society},
journal={Bull. Lond. Math. Soc. },
volume = {21},
number = {2},
pages = {113-152},
doi = {https://doi.org/10.1112/blms/21.2.113},
url = {https://londmathsoc.onlinelibrary.wiley.com/doi/abs/10.1112/blms/21.2.113},
eprint = {https://londmathsoc.onlinelibrary.wiley.com/doi/pdf/10.1112/blms/21.2.113},
year = {1989}
}

@article{FollandSitaram1997TheUncertaintyPrinciple,
 author = {Folland, Gerald B. and Sitaram, Alladi},
 title = {The uncertainty principle: {A} mathematical survey},
 fjournal = {The Journal of Fourier Analysis and Applications},
 journal = {J. Fourier Anal. Appl.},
 issn = {1069-5869},
 volume = {3},
 number = {3},
 pages = {207--238},
 year = {1997},
 language = {English},
 doi = {10.1007/BF02649110},
 url = {https://eudml.org/doc/59507},
 zbMATH = {1038862},
 Zbl = {0885.42006}
}

@article{CarneiroQuesada-Herrera2023GeneralizedSign,
 author = {Carneiro, Emanuel and Quesada-Herrera, Emily},
 title = {Generalized sign {Fourier} uncertainty},
 fjournal = {Annali della Scuola Normale Superiore di Pisa. Classe di Scienze. Serie V},
 journal = {Ann. Sc. Norm. Super. Pisa, Cl. Sci. (5)},
 issn = {0391-173X},
 volume = {24},
 number = {3},
 pages = {1671--1704},
 year = {2023},
 language = {English},
 doi = {10.2422/2036-2145.202105_026},
 zbMATH = {7771794},
 Zbl = {1527.42007}
}

@article{GoncalvesOliveiraESilvaRamos2021-OnRegularityAndMass,
 author = {Gon{\c{c}}alves, Felipe and Oliveira e Silva, Diogo and Ramos, Jo{\~a}o P. G.},
 title = {On regularity and mass concentration phenomena for the sign uncertainty principle},
 fjournal = {The Journal of Geometric Analysis},
 journal = {J. Geom. Anal.},
 issn = {1050-6926},
 volume = {31},
 number = {6},
 pages = {6080--6101},
 year = {2021},
 language = {English},
 doi = {10.1007/s12220-020-00519-7},
 zbMATH = {7379198},
 Zbl = {1469.42015}
}

@article{CohnGoncalves2019-AnOptimalUncertaintyPrinciple,
 author = {Cohn, Henry and Gon{\c{c}}alves, Felipe},
 title = {An optimal uncertainty principle in twelve dimensions via modular forms},
 fjournal = {Inventiones Mathematicae},
 journal = {Invent. Math.},
 issn = {0020-9910},
 volume = {217},
 number = {3},
 pages = {799--831},
 year = {2019},
 language = {English},
 doi = {10.1007/s00222-019-00875-4},
 zbMATH = {7089445},
 Zbl = {1422.42010}
}

@article{Viazovska2017-TheSpherePackingProblemInDimension8,
 author = {Viazovska, Maryna S.},
 title = {The sphere packing problem in dimension 8},
 fjournal = {Annals of Mathematics. Second Series},
 journal = {Ann. Math. (2)},
 issn = {0003-486X},
 volume = {185},
 number = {3},
 pages = {991--1015},
 year = {2017},
 language = {English},
 doi = {10.4007/annals.2017.185.3.7},
 zbMATH = {6731863},
 Zbl = {1373.52025}
}

@book{AmbrosioGigliSavare2008-GradientFlows,
 author = {Ambrosio, Luigi and Gigli, Nicola and Savar{\'e}, Giuseppe},
 title = {Gradient flows in metric spaces and in the space of probability measures},
 edition = {2nd ed.},
 isbn = {978-3-7643-8721-1},
 year = {2008},
 publisher = {Basel: Birkh{\"a}user},
 language = {English},
 zbMATH = {5233008},
 Zbl = {1145.35001}
}

@misc{deBranges1968-HilbertSpacesOfEntireFunctions,
 author = {de Branges, Louis},
 title = {Hilbert spaces of entire functions},
 year = {1968},
 language = {English},
 howpublished = {Prentice-{Hall} {Series} in {Modern} {Analysis}. {Englewood} {Cliffs}, {N}.{J}.: {Prentice}-{Hall}. ix, 326 p. 102 s. 6d. (1968).},
 zbMATH = {3253423},
 Zbl = {0157.43301}
}

@article{GoncalvesOliveira-e-SilvaSteinerberger-HermitePolynomialsLinearFlows,
 author = {Gon{\c{c}}alves, Felipe and Oliveira e Silva, Diogo and Steinerberger, Stefan},
 title = {Hermite polynomials, linear flows on the torus, and an uncertainty principle for roots},
 fjournal = {Journal of Mathematical Analysis and Applications},
 journal = {J. Math. Anal. Appl.},
 issn = {0022-247X},
 volume = {451},
 number = {2},
 pages = {678--711},
 year = {2017},
 language = {English},
 doi = {10.1016/j.jmaa.2017.02.030},
 zbMATH = {6790498},
 Zbl = {1373.33018}
}

\end{document}